\documentclass[a4paper, 10pt]{amsart}
\usepackage[top=80pt,bottom=65pt,left=70pt,right=70pt]{geometry}
\usepackage{graphicx, eepic}
\usepackage{times}
\normalfont
\usepackage{bbm}
\usepackage{xcolor}
\usepackage{verbatim}
\usepackage{kotex}
\usepackage{mathrsfs}
\usepackage{amscd}
\usepackage{float}
\usepackage[all]{xy}
\usepackage[subnum]{cases}
\ifpdf
\usepackage{tikz}
\usepackage{textcomp}
\fi
\usetikzlibrary{arrows,matrix,positioning}

\usepackage{amsthm,amsmath,amsfonts,amssymb}
\usepackage[colorlinks=true, linkcolor=blue, citecolor=red, filecolor=violet, urlcolor=violet, backref=page]{hyperref}

\usepackage{multirow, url}
\usepackage{color}
\usepackage[all]{xy}
\theoremstyle{plain}
\newtheorem{theorem}{Theorem}[section]

\newtheorem{thmx}{Theorem}

\newtheorem{proposition}[theorem]{Proposition}
\newtheorem{lemma}[theorem]{Lemma}

\newtheorem{corollary}[theorem]{Corollary}

\theoremstyle{definition}

\newtheorem{definition}[theorem]{Definition}

\newtheorem{example}[theorem]{Example}

\theoremstyle{remark}
\newtheorem{remark}[theorem]{Remark}

\renewcommand{\bar}{\overline}

\renewcommand{\vec}{\overrightarrow}

\newcommand{\dge}{\rotatebox[origin=c]{45}{$\ge$}}
\newcommand{\uge}{\rotatebox[origin=c]{315}{$\ge$}}

\newcommand{\updots}{\hbox to1.65em{\rotatebox[origin=c]{45}{$\cdots$}}}
\newcommand{\dndots}{\hbox to1.65em{\rotatebox[origin=c]{315}{$\cdots$}}}

\newcommand{\C}{\mathbb{C}}

\newcommand{\R}{\mathbb{R}}

\newcommand{\Z}{\mathbb{Z}}
\newcommand{\p}{\mathbb{P}}
\newcommand{\F}{\mathcal{F}}
\newcommand{\CP}{\mathbb{C}P}

\def\pa{\partial}
\def\mcal{\mathcal}
\def\frak{\mathfrak}

\newcommand{\vs}{\vspace}
\newcommand{\ds}{\displaystyle}

\numberwithin{equation}{section} \numberwithin{table}{section}

\begin{document}                                                                          

\title{Monotone Lagrangians in flag varieties}
\author{Yunhyung Cho}
\address{Department of Mathematics Education, Sungkyunkwan University, Seoul, Republic of Korea. }
\email{yunhyung@skku.edu}

\author{Yoosik Kim}
\address{Department of Mathematics, Brandeis University, Waltham, USA and Center of Mathematical Sciences and Applications, Harvard University, Cambridge, USA}
\email{yoosik@brandeis.edu, yoosik@cmsa.fas.harvard.edu}


\begin{abstract}
In this paper, we give a formula for the Maslov index of a gradient holomorphic disc, which is a relative version of the Chern number formula of a gradient holomorphic sphere for a Hamiltonian $S^1$-action. Using the formula, we classify all monotone Lagrangian fibers of Gelfand--Cetlin systems on partial flag manifolds.
\end{abstract}
\maketitle
\setcounter{tocdepth}{1} 
\tableofcontents

\section{Introduction}
\label{secIntroduction}

	Every Lagrangian submanifold $L$ in a $2n$-dimensional symplectic manifold $(M,\omega)$
	comes up with the so-called Maslov homomorphism $\mu \colon \pi_2(M,L) \to \Z$. 
	The output $\mu(\beta)$ of a homotopy class $\beta$ is called the {\em Maslov index} of $\beta$, which can be thought as a relative version of the Chern number
	of a spherical class in $\pi_2(M)$. 
	The Maslov index is particularly important in the theory of moduli spaces of pseudo-holomorphic curves from bordered Riemann surfaces as it involves the Fredholm index of 
	a linearization of the
	Cauchy--Riemann operator via the Riemann--Roch theorem.
	In particular, for the purpose of studying Lagrangian Floer theory on symplectic toric manifolds, Cho \cite{Cho} and Cho--Oh \cite{CO} introduced a formula for the Maslov index 
	of a holomorphic disc bounded by a Lagrangian toric fiber. 
	The Maslov index is twice the intersection number of the disc and the toric anti-canonical divisor. 
	In a more general context, Auroux \cite{Aur} derived a Maslov index formula 
	for a special Lagrangian submanifold
	in the complement of an anti-canonical divisor of a K\"{a}hler manifold. 
	Those formulae are crucially used for classifying the holomorphic discs, which leads to a mirror Landau--Ginzburg 
	model arising from deformations of Floer theory in Fukaya--Oh--Ohta--Ono
	 \cite{FOOO, FOOOToric1}.
	
	The first goal of this paper is to deduce a formula for the Maslov index of a {\em gradient disc}, which is an analogue of a {\em gradient sphere} in Karshon \cite{Ka}, 
	bounded by an $S^1$-invariant Lagrangian submanifold 
	in a symplectic manifold admitting a Hamiltonian $S^1$-action. 
	
	\begin{thmx}[Theorem \ref{theorem_Maslov_index_formula}]\label{thmx_Maslovindex}
		Let $(M,\omega)$ be a $2n$-dimensional symplectic manifold equipped with an effective Hamiltonian $S^1$-action with a moment map 
		$H \colon M \rightarrow \R$. Suppose that $L$ is an $S^1$-invariant Lagrangian 
		submanifold of $(M,\omega)$ lying on some level set of $H$. 
		For any gradient holomorphic disc $u \colon (\mathbb{D}, \partial \mathbb{D}) \rightarrow (M, L)$, we then have 
		\[
			\mu([u]) = -2n_z
		\]	
		where $n_z$ is the sum of weights at the unique fixed point $z$ in $u(\mathbb{D})$.
		In particular, if the action is semifree and $H(z)$ is the maximum, then the Maslov index $\mu([u])$ equals the codimension of the maximal fixed component of the action.
	\end{thmx} 
	
	This formula can be understood as a relative version of the Chern number formula of a gradient sphere in Ahara--Hattori \cite{AH}. Indeed, we use Lerman's symplectic cut \cite{L} to reduce our  
	gradient disc to the gradient sphere. 	
	
	It is worth mentioning that Theorem \ref{thmx_Maslovindex} can be applied to a symplectic manifold having a locally defined Hamiltonian $S^1$-action. 
	More precisely, if a Hamiltonian $S^1$-action is defined on an open subset $U$ of $(M,\omega)$, then one can apply Theorem \ref{thmx_Maslovindex} 
	to a gradient disc as long as the image of the disc is fully contained in $U$. It exactly fits into the situation of \emph{Gelfand--Cetlin systems} on partial flag manifolds for instance. 
	
	A {\em Gelfand--Cetlin system}, or shortly a {\em GC system}, 
	is a completely integrable system on a partial flag manifold constructed by Guillemin and Sternberg \cite{GS}. 
	The image is a convex polytope 
	$\Delta$, which is called a {\em Gelfand--Cetlin polytope}, or a {\em GC polytope} for short.
	As the big torus action does not extend to the ambient manifold, non-torus Lagrangian fibers can appear over a lower dimensional face of the polytope $\Delta$. 
 	For the case of partial flag manifolds of type A, the authors with Oh \cite{CKO1} locate the non-torus GC fibers and describe their topology. 
	One consequence of \cite[Theorem A]{CKO1} is that the fiber over a point in the relative interior of a face $f$ is Lagrangian if and only if the fiber over any point in the relative interior of 
	$f$ is also Lagrangian. 
 	In this regard, a face $f$ of $\Delta$ is said to be \emph{Lagrangian} if one fiber over its relative interior point is Lagrangian (and hence all). 
 	
 	Using the formula, we classify all \emph{monotone} Lagrangian fibers of GC systems. A Lagrangian submanifold is said to be \emph{monotone} if the symplectic area of discs are positively
 	 proportional to their Maslov index, that is, for some positive real number $c > 0$, 
 	 \[
 	 	\omega(\beta) = c \cdot \mu(\beta), \quad \beta \in \pi_2(M,L).
 	 \] 
 	 The notion of monotone Lagrangian submanifolds was introduced by Oh \cite{Oh} 
 	 as a nice condition for constructing Lagrangian Floer homology. 
 	 Our second main theorem states that the GC fiber at the {\em center} \footnote{See~\eqref{equ_aibi} in Section  \ref{secClassificationOfMonotoneLagrangianFibers} for precise description of the center of a Lagrangian face.} of a Lagrangian face is monotone and each monotone Lagrangian GC fiber is located at the center of a Lagrangian face.

	\begin{thmx}[Theorem \ref{theorem_main}]\label{thmx_main}
		Consider a partial flag manifold equipped with a monotone\footnote{A symplectic form $\omega$ on $M$ is called {\em monotone} if the cohomology class
		$[\omega] \in H^2(M;\R)$ is positively proportional to $c_1(TM)$ with respect to some (any) $\omega$-compatible almost complex structure on $M$, that is, $c_1(TM) = c \cdot [\omega]$ for 
		some $c > 0$. See Section~\ref{ssecMonotoneLagrangians}.}
		Kirillov--Kostant--Souriau symplectic form
		and let $\Delta$ be the corresponding GC polytope. 
		For a point $\textbf{\textup{u}} \in \Delta$, the fiber of the GC system at $\textbf{\textup{u}}$ is monotone Lagrangian if and only if $\textbf{\textup{u}}$ is the center of a Lagrangian face of $\Delta$.
	\end{thmx}
		
	As a special case of Theorem \ref{thmx_main}, it immediately follows that 
	the GC torus fiber at $\textbf{\textup{u}}$ is monotone if and only if $\textbf{\textup{u}}$ is the center of $\Delta$ since $\Delta$ is itself the (unique) improper Lagrangian face of $\Delta$. 
	
	We hope that our classification of monotone Lagrangian GC fibers would serve as the base step toward understanding the (monotone) Fukaya category and mirror symmetry of partial flag varieties.  
	In the light of the work of Nohara--Ueda \cite{NU2} and Evans--Lekili \cite{EL} proving that certain monotone GC fibers (together with deformation data) in certain Grassmannians split-generates the Fukaya category, monotone 
	Lagrangian GC fibers are candidates for non-zero objects of the Fukaya category over a ring (with certain characteristic). 
	Also, a preferred Landau--Ginzburg mirror constructed by Rietsch \cite{Ri} in the setting of closed mirror symmetry is defined on a partial compactification of algebraic torus together with a holomorphic function on it. 
	Non-torus monotone Lagrangians are presumably in charge of the partial compactification of the Landau--Ginzburg mirror of torus fiber.
		
	The paper is organized as follows. 
	In Section~\ref{secGradientJHolomorphicDisks}, we define a gradient holomorphic disc generated by a Hamiltonian $S^1$-action. 
	Section~\ref{secMaslovIndicesAndChernNumbers} discusses the Maslov index for gradient discs and proves Theorem \ref{thmx_Maslovindex}.
	In Section~\ref{secGelfandCetilnSystems}, we review GC systems and recall some results in \cite{CKO1}.
	Section~\ref{secClassificationOfMonotoneLagrangianFibers} is devoted to classifying the monotone Lagrangian GC fibers and to proving Theorem~\ref{thmx_main}.

\subsection*{Acknowledgements} 
The authors would like to thank Cheol-Hyun Cho, Yong-Geun Oh, Kaoru Ono for helpful discussions and the anonymous referees for their detailed comments which helped to improve the manuscript.
The first author is supported by the National Research Foundation of Korea(NRF) grant funded by the Korea government(MSIP; Ministry of Science, ICT \& Future Planning) (NRF-2017R1C1B5018168). This project was initiated when two authors were affiliated to IBS-CGP and were supported by IBS-R003-D1.

\section{Gradient $J$-holomorphic discs}
\label{secGradientJHolomorphicDisks}

In this section, we introduce the notion of a {\em gradient disc}, an analogue of a gradient sphere (cf. \cite{AH, Au, Ka}), in a Hamiltonian $S^1$-manifold. 

Let $(M,\omega)$ be a symplectic manifold. Assume that the unit circle group $S^1$ acts effectively on $M$ and denote by $\xi$ the vector field on $M$ generated by the $S^1$-action.
The action is said to be {\em Hamiltonian} if 
\begin{equation}\label{equation_Hamiltonian}
	\iota_\xi \omega = -dH
\end{equation}	
for some smooth function $H$ on $M$. Such a function $H$ is called a {\em moment map} for the $S^1$-action, 
or a {\em periodic Hamiltonian}. It is an immediate consequence that $x \in M$ is a critical point of $H$ if and only if it is a fixed point of the $S^1$-action.  
(See also \cite[Remark II.3.1]{Au}.) We also have the following.

\begin{lemma}[Proposition 2.9 in \cite{GGK}]\label{lemma_fixedpoint_criticalpoint_equal}
	A moment map $H$ is constant on any $S^1$-orbit. 
\end{lemma}

Let $J$ be an $S^1$-invariant $\omega$-compatible almost complex structure on $(M,\omega)$. 
Note that such a $J$ always exists, see \cite[II.2, IV.1.b]{Au} for instance.
With respect to the Riemannian metric $g_J(\cdot, \cdot) := \omega(J\cdot, \cdot)$,
the gradient vector field $\nabla H$ of $H$ is characterized by 
\begin{equation}\label{equation_nabla}
	g_J(\nabla H, Y) = dH(Y), \quad \text{or equivalently,} \quad \omega(J\nabla H, Y) = -\omega(\xi, Y)
\end{equation}
for every vector field $Y$ on $M$. Therefore we have $\nabla H = J\xi$. 
Assuming the completeness of the vector field $J\xi$ on $M$,
the one-parameter subgroup action
\begin{equation}\label{equation_gradient_flow}
	\begin{array}{cccl}
		\gamma \colon   & \R \times M & \rightarrow & M \\
				     &  (s, q) & \mapsto & \gamma_s(q)
	\end{array}
\end{equation}
is well-defined on $\R$ where $\gamma_s$ is an integral curve of $J\xi$ defined by the following differential equation
\[
\begin{cases}
	\frac{d}{ds} \gamma_s(q)  = J\xi(\gamma_s(q))  \\
	\gamma_0(q) = q
\end{cases}
\]
for every $q \in M$.

Now, pick a point $p \in M$ whose stabilizer is the trivial subgroup of $S^1$ and let $\sigma_p := S^1 \cdot p$ be the free $S^1$-orbit containing $p$. 
Since $J$ and $\xi$ are both $S^1$-invariant, so is $J\xi$ and therefore we can define a map 
\begin{equation}\label{equation_u}
	\begin{array}{cccl}
		u \colon  & \R_{\geq 0} \times S^1 & \rightarrow & M \\
		     &  (s, t) & \mapsto & \gamma_s(t \cdot p),
	\end{array} 
\end{equation} 
whose image consists of $S^1$-orbits.

\begin{lemma}\label{lemma_gradientflow_converge}
	Suppose that $\gamma_s(p)$ converges to some fixed point $z_p \in M^{S^1}$ as $s$ goes to infinity.
	Then every point in $\sigma_p$ converges to the point $z_p$ along $\gamma_s$. In other words, 
	\[
		\lim_{s \rightarrow \infty} \gamma_s(q) = z_p
	\] for every $q \in \sigma_p$.
\end{lemma}

\begin{proof}
	Fix $q \in \sigma_p$. 
	Since the one-parameter group action $\gamma$ commutes with the $S^1$-action, we have 
	\[
		\lim_{s \rightarrow \infty} t \cdot \gamma_s(q) = \lim_{s \rightarrow \infty} \gamma_s(t \cdot q) = \lim_{s \rightarrow \infty} \gamma_s(p) = z_p.
	\]
	where $t \in S^1$ such that $t \cdot q = p$. 
\end{proof}

\begin{lemma}\label{lemma_holomorphic_cylinder}
	The map $u$ in \eqref{equation_u} is $(j,J)$-holomrophic where $j$ is the almost complex structure on $\R \times S^1$ given by
	\[
		j \cdot \frac{\partial}{\partial t} = \frac{\partial}{\partial s}, \quad j \cdot \frac{\partial}{\partial s} = -\frac{\partial}{\partial t}.
	\]
\end{lemma}

\begin{proof}
	Comparing 
	\[
		J\xi\left(u(s_0,t_0) \right) = J\xi(\gamma_{s_0}(t_0 \cdot p))
		= \left.\frac{\partial}{\partial s}\right|_{s=s_0} \gamma_s(t_0 \cdot p) 
		= \frac{\partial u}{\partial s}(s_0, t_0) = du_{(s_0,t_0)} \left(\frac{\partial}{\partial s} \right) = du \left( j \cdot \frac{\partial}{\partial t} \right)
	\] with 
	\[
		\xi \left(u(s_0, t_0) \right) = \xi(\gamma_{s_0}(t_0 \cdot p)) = \left.\frac{\partial}{\partial t}\right|_{t=t_0} \gamma_{s_0} (t \cdot p) 
		= \frac{\partial u}{\partial t}(s_0, t_0) = du_{(s_0,t_0)} \left(\frac{\partial}{\partial t} \right) = du \left(\frac{\partial}{\partial t} \right),
	\]
	we obtain 
	\[
		du \left( j \cdot \frac{\partial}{\partial t} \right) = J \circ du \left(\frac{\partial}{\partial t} \right). 
	\] 
	Similarly, we see that  
	$ du \left( j \cdot \frac{\partial}{\partial s} \right) = J \circ du \left(\frac{\partial}{\partial s} \right)$ and this completes the proof.
\end{proof}

Recall that the half-infinite cylinder $\R_{\geq 0} \times S^1$ is conformally equivalent to the punctured unit disc $\mathbb{D} \setminus \{0\}$ in $\C$. 
If $\gamma_s(p)$ converges to some fixed point $z_p$,
by Lemma~\ref{lemma_gradientflow_converge}, the map $u$ can be extended to a continuous map $\widetilde{u} \colon \mathbb{D} \rightarrow M$. Furthermore, the $L^2$-norm (called the {\em energy of $u$}) of the derivative of a $J$-holomorphic curve satisfies 
\[
	\int_\mathbb{D} ||du||^2 = \int_\mathbb{D} u^*\omega = H(z_p) - H(p),
\]
where the latter equality is known to Archimedes. 
(See \cite[Lemma 4.1.2]{MS2} and \cite[Corollary 30.4]{Ca}.)
In particular, since the energy of $u$ is bounded (uniformly), we apply the Riemann extension theorem to confirm that $u$ can be extended to a $J$-holomorphic map $\widetilde{u} \colon \mathbb{D} \rightarrow M$.

\begin{definition}\label{definition_gradient_holo_disk}
	The extended map $\widetilde{u} \colon \mathbb{D} \rightarrow M$ is called a {\em gradient $J$-holomorphic disc} 
	of an $S^1$-orbit $\sigma_p$ and is denoted by $u_p^J$.
\end{definition}

\vspace{-0.5cm}
	\begin{figure}[H]
		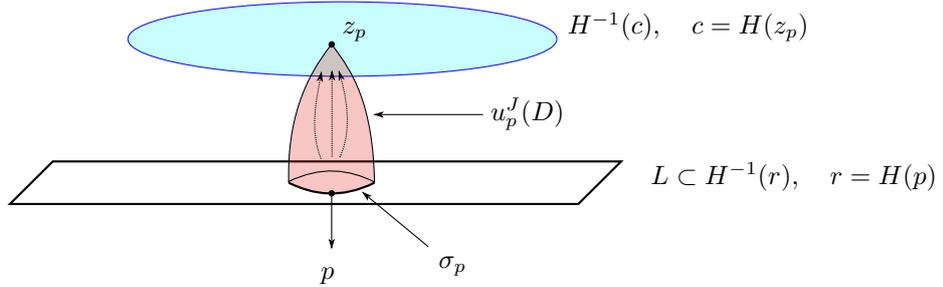
		\caption{\label{figure_grad_hol_disk} Gradient holomorphic disc}
	\end{figure}
\vspace{-0.5cm}

\begin{remark}
The map $u$ in~\eqref{equation_u} cannot be only defined on $\R_{\geq 0} \times S^1$, but defined on the infinite cylinder $\R \times S^1$ by considering the gradient flow with respect to $g_{-J}$. 
Under the assumption that $\gamma_s$ converges to a fixed point as $s$ goes to $\pm \infty$, the map $u$ can be extended to a smooth map defined on $S^2$.
It is called a {\em gradient sphere} in \cite{Au, AH, Ka}. Note that a gradient sphere containing an $S^1$-orbit $\sigma_p$ can be obtained by gluing two gradient holomorphic discs
$u_p^J$ and $u_p^{-J}$ along their common boundary $\sigma_p$.
\end{remark}

\vspace{0.1cm}
\section{Maslov index formula }
\label{secMaslovIndicesAndChernNumbers}

For any continuous map $u \colon (\mathbb{D}, \pa \mathbb{D}) \to (M, L)$, any trivialization $u^*TM \cong \mathbb{D} \times \C^n$ 
as a symplectic vector bundle defines a loop $\ell : S^1 \rightarrow \Lambda(n)$ in the Lagrangian Grassmannian $\Lambda(n) \simeq U(n) / O(n)$ of $\C^n$ 
induced by $(u|_{\partial \mathbb{D}})^*TL$.
Then the {\em Maslov index} of $u$ is defined as the degree of the map
\[
	{\det}^2 \circ \ell \colon S^1 \to U(n) / O(n) \to S^1.
\]
It turned out that the Maslov index is independent of the choice of a trivialization and is well-defined up to homotopy, and hence it is defined on $\pi_2(M, L)$.
We denote by $\mu([u])$ the Maslov index of $u$. 

The aim of this section is to derive a formula for the Maslov index of a gradient holomorphic disc. 
More precisely, we will show that the Maslov index of a gradient $J$-holomorphic disc $u \colon \mathbb{D} \rightarrow M$ 
bounded by an $S^1$-invariant Lagrangian submanifold having constant momentum
is determined by the weights of the tangential $S^1$-representation at the unique fixed point contained in $u(\mathbb{D})$. 

We begin with the following.

\begin{lemma}\label{lemma_isotopy}
	Let  $L$ be an $n$-dimensional manifold and $(M,\omega)$ be a $2n$-dimensional symplectic manifold.
	Suppose that there is a Lagrangian isotopy $\varphi \colon  [0,1]  \times L \to M$, that is, a smooth map such that 
	$\varphi_s := \varphi (s, \, \cdot \,) \colon L \to M$ is a Lagrangian embedding
	for each $s \in [0,1]$. 
	We further assume that $\varphi$ is an embedding.
	Then the Maslov index is preserved through the isotopy. 
	More precisely, for any disc $u_0 \colon (\mathbb{D}, \pa \mathbb{D}) \to (M, \varphi_0 (L))$, 
	consider the (extended) disc $u_s \colon (\mathbb{D}, \pa \mathbb{D}) \to (M, \varphi_s (L))$ 
	obtained by gluing the disc $u_0$ and the cylinder 
	\[
		\begin{array}{ccccl}\vspace{0.1cm}
		 	v \colon &  [0,s] \times \pa \mathbb{D} & \to & M \\ \vspace{0.1cm}
		 	& (s', z) & \mapsto & \varphi_{s'}(\varphi_0^{-1}(u_0(z))) 
		\end{array}
	\] 
	along the boundary of the disc $u_0 (\mathbb{D})$.
	Then, we have 
	\[
		\mu ([u_0]) = \mu ([u_s])
	\]
	where $[u_s] \in \pi_2(M, \varphi_s(L))$ is the homotopy class represented by $u_s$ and $\mu([u_s])$ denotes its
	Maslov index.
\end{lemma}

\begin{figure}[H]
	\scalebox{1}{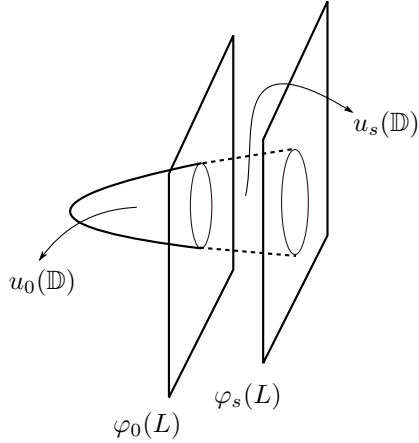}
	\caption{\label{figure_isotopy} Extending disc by Lagrangian isotopy}
\end{figure}

\begin{proof}
	Note that $\varphi^* \omega = ds \wedge \alpha$ where $\alpha$ is a 
	1-forms on $[0,1]\times L$ because $\varphi^* \omega$ vanishes on  $\{ \mathrm{pt} \} \times L$. 
	Moreover, each one-form $\alpha_s := \alpha|_{ \{s\} \times L}$ is closed since $d \varphi^* \omega = ds \wedge d\alpha = 0$, 
	see \cite[Sect. 6.1]{Pol}. Extend $\alpha_s$ to $[0,1] \times L$ such that $\alpha_s |_{\{s'\} \times L} = \alpha_s$ for every $s' \in [0,1]$ and 
	still denote the extended (closed) 1-form (on $[0,1]\times L$) by $\alpha_s$.
	
	Since $\varphi$ is an embedding, we can take a sufficiently small open neighborhood $\mcal{U}$ of $\varphi([0,1] \times L)$ 
	which deformation retracts to $\varphi([0,1] \times L)$
	in $M$ with the retraction $\pi \colon \mcal{U} \rightarrow \varphi([0,1] \times L)$. 
	Then we obtain a family of closed 1-forms
	$\left\{ \widetilde{\alpha}_s := \pi^* \left(( \varphi^{-1})^* \alpha_s \right) \right\}_{s \in [0,1]}$
	on $\mcal{U}$ and it satisfies
	\[
		\varphi^* \left( \widetilde{\alpha}_s \right) = \varphi^* \left( \left.\widetilde{\alpha}_s\right|_{\varphi([0,1]\times L)} \right) 
		= \varphi^* \left( \varphi^{-1} \right)^* \alpha_s
		= \alpha_s, \quad \quad s \in [0,1]
	\]
	where we denote by $\varphi^{-1}$ the inverse of $\varphi$ from $\varphi([0,1] \times L)$ to $[0,1] \times L$.
	
	We claim the (time-dependent) vector field $X_s$ given by $\iota_{X_s} \omega = \widetilde{\alpha}_s$ on $\mcal{U}$ generates the 
	symplectic isotopy that coincides with the Lagrangian isotopy on $\varphi_0(L)$ given by $\varphi$. 
	To show this, note that
	\begin{itemize}
		\item $\mcal{L}_{X_s} \omega = \iota_{X_s} d\omega + d \iota_{X_s} \omega = d \widetilde{\alpha}_s = 0$ (so that $X_s$ is symplectic), and 
		\item $\varphi^*\left(\omega_{\varphi(s, x)} \left( \frac{d}{ds} \varphi(s, x), \cdot \right)\right) = 
		(\varphi^*\omega)_{(s,x)} \left( \frac{\partial}{\partial s}, \cdot \right) = \alpha_{(s,x)} = (\alpha_s)_x$.
	\end{itemize}
	Therefore, we have 
	\[
		\omega_{\varphi(s, x)} \left( \frac{d}{ds} \varphi(s, x), \cdot \right) = (\widetilde{\alpha}_s)_{\varphi(s,x)} = \omega_{\varphi(s,x)}(X_s, \cdot)
		  \quad \text{on} \quad \varphi([0,1] \times L). 
	\]
	That is, $\frac{d}{ds} \varphi(s, x) = X_s(\varphi(s, x))$ since $\widetilde{\alpha}_s$ vanishes on the vertical tangent bundle over $\varphi([0,1] \times L)$ 
	with respect to the retraction $\pi$. This proves the claim. 

	Now we are ready to show $\mu ([u_0]) = \mu ([u_s])$. Let $\mcal{V}$ be a closed collar neighborhood of $\partial \mathbb{D}$ in $\mathbb{D}$ such that 
	$u_s(\mcal{V}) \subset \mcal{U}$ for every $s \in [0,1]$. (This is always possible since $[0,1]$ is compact.) Then $\mcal{V}$ becomes a bordered
	Riemann surface diffeomorphic to an annulus and the pull-back bundle 
	\[
		(u_0|_{\mcal{V}})^* TM \cong \mcal{V} \times \C^n
	\] 
	is trivial and has a symplectic trivialization. 
	We can identify each $(u_s|_{\mcal{V}})^* TM$ with $(u_0|_{\mcal{V}})^* TM$ symplectically via the symplectic isotopy on $\mcal{U}$ generated by $X_s$. 
	Moreover, each $u_s(\partial \mathbb{D})$ represents a loop in the Lagrangian Grassmannian $\Lambda(\C^n)$ (with respect to the trivialization that we have chosen) and the symplectic isotopy
	induces a homotopy from $u_0(\partial \mathbb{D})$ to $u_s(\partial \mathbb{D})$. Consequently, $u_0$ and $u_s$ have the same Maslov index. 	
\end{proof}

\begin{remark}
	The invariance of the Maslov index under Lagrangian isotopy seems to be well-known to experts (the first statement in the proof of Lemma 2.8 in \cite{LM}, for instance). 
	However, we could not find any literature containing the precise statement that we need and its proof. (For exact Lagrangian isotopy case, see \cite[Section 2.1]{Oh2}.)
\end{remark}

Consider a $2n$-dimensional closed symplectic manifold $(M,\omega)$ endowed with an effective Hamiltonian $S^1$-action.
Let $\xi$ be the vector field generated by the action and $H \colon M \rightarrow \R$ a moment map. 
For a point $p \in M$ with $H(p) = r$, consider an $S^1$-invariant Lagrangian submanifold $L$ containing $p$ and contained in the level set $H^{-1}(r)$.
For an $\omega$-compatible $S^1$-invariant almost complex structure $J$, the gradient vector field of $H$ with respect to the metric $g_J$ equals $J\xi$ as computed in ~\eqref{equation_nabla}. 
For any $s \in \mathrm{Im}~H$ that is not an extreme value of $H$, we denote by $M_s$ the quotient space $H^{-1}(s) / S^1$ and by $\omega_s$ the reduced symplectic form on the smooth locus of $M_s$.
The following lemma tells us that a neighborhood of a free $S^1$-orbit $\sigma_a$ in $L$ can be identified with that of another free $S^1$-orbit as long as the orbits are related by the gradient 
flow generated by $J\xi$.

\begin{lemma}\label{lemma_flowing_lagrangian}
	Let $[a,b]$ be the closed interval in $\mathrm{Im} ~H$ not containing any extreme values of $H$ and let $\sigma_a$ a free $S^1$-orbit in $H^{-1}(a)$.
	Assume that the orbit $\sigma_a$ flows into a free orbit $\sigma_b$ through free orbits along the gradient vector field of $H$.
	For any $S^1$-invariant Lagrangian submanifold $L \subset H^{-1}(a)$ containing $\sigma_a$, there exists an $n$-dimensional manifold $\mcal{L}$ 
	and an embedding 
	\[
			\varphi \colon  [a,b] \times \mcal{L}  \rightarrow  M, \quad
					       (s,z)  \mapsto  \varphi_s(z)
	\]
	obeying   
	\begin{itemize}
		\item for each $s \in [a,b]$, $\varphi_s(\mcal{L})$ is an $S^1$-invariant Lagrangian submanifold of $(M,\omega)$,
		\item for each $s \in [a,b]$, $\varphi_s(\mcal{L})$ is in the level set $H^{-1}(s)$,
		\item $\varphi_a (\mcal{L})$ is a tubular neighborhood of $\sigma_a$ in $L$, and
		\item $\varphi_b(\mcal{L})$ is a Lagrangian submanifold of $M$ containing the orbit $\sigma_b$.
	\end{itemize}
In particular, we have a Lagrangian isotopy $\varphi$ from $\varphi_a (\mcal{L})$ to $\varphi_b (\mcal{L})$ and satisfies the condition in Lemma \ref{lemma_isotopy} (since $\varphi$ 
is an embedding.)
\end{lemma}

\begin{proof}
Recall that for any point $q \in M$, $\gamma_s(q) := \gamma(s,q) \colon \R \rightarrow M$ is a parametrized integral curve of $J\xi$ such that $\gamma(0, q) = q$ where $\gamma$ is defined in 
\eqref{equation_gradient_flow}. 

Let $\sigma_b \subset H^{-1}(b)$ be the $S^1$-orbit into which the given free orbit $\sigma_a$ flows along the gradient vector field $J\xi$.
Fix a point $p \in \sigma_a$. 
By reparametrizing the variable $s$ if necessary, we may assume that $\gamma_s (p) \in H^{-1}(s)$ for every
$s \in [a,b]$. 
Consider an $S^1$-equivariant embedding {of a closed cylinder} 
\[
\Gamma \colon [a,b] \times S^1 \to M \quad (s, t) \mapsto \gamma_{s} (t \cdot p).
\]

Now we take an $S^1$-invariant neighborhood $\mcal{V}$ of the cylinder $\mathrm{Im}~\Gamma$ in $H^{-1}([a,b])$ sufficiently small so that the induced $S^1$-action is still free on $\mcal{V}$. 
The quotient $\underline{\mcal{V}} := \mcal{V} / S^1$ can be thought as a parametrized (by $s$) 
family of open neighborhoods of each point $[\sigma_s]$ in the quotient space $M_s$.
Moreover, we may choose $\mcal{V}$ such that $\underline{\mcal{V}}$ is contractible, 
and therefore we may assume that $\mcal{V} = \underline{\mcal{V}} \times S^1$ where the $S^1$-action on $\mcal{V}$ is just the rotation of $S^1$ on the second factor of 
$\underline{\mcal{V}} \times S^1$.

Since the level $[a,b]$ is compact, we may choose a sufficiently small open ball $\mcal{W} \subset \R^{2n-2}$ centered at the origin $O$ of $\R^{2n-2}$ such that there is a symplectic embedding 
\[
	\iota_s \colon (\mcal{W}, \omega_{\R^{2n-2}}) \hookrightarrow (M_s, \omega_s), 
\]
satisfying $O \mapsto [\sigma_s]$ for each $s \in [a,b]$. Here, $\omega_{\R^{2n-2}}$ is the standard symplectic form on $\mcal{W} \subset \R^{2n-2}$.
We then have a smooth embedding
\[
	\psi \colon [a, b] \times (\mcal{W}, \omega_{\R^{2n-2}}) \to \underline{\mcal{V}} \subset \bigcup_{a \leq s \leq b} (M_s, \omega_s) 
\]
such that 
\begin{itemize}
\item $\psi (s, \cdot) = \iota_s$, and 
\item $\psi_s (O) = [\sigma_s]$ for each $s \in [a,b]$. 
\end{itemize}
 
Consider any $S^1$-invariant Lagrangian submanifold $L$ of $(M,\omega)$ in the level set $H^{-1}(a)$ containing the orbit $\sigma_a$.
Setting $\underline{L}$ to be the quotient of $L$ by the $S^1$-action, take a tubular neighborhood $\underline{{U}}$ of $[\sigma_a]$ in $\underline{L}$, 
which is a Lagrangian submanifold of $(M_a, \omega_a)$. By taking the neighborhood $\underline{{U}}$ small enough, we may assume that its closure is contained in $\psi_a (\mcal{W})$.
Denote by $\underline{\mcal{L}}$ the Lagrangian submanifold $\psi_a^{-1}(\underline{{U}})$ of $(\mcal{W}, \omega_{\R^{2n-2}})$.
We then have a family $\{ \psi_s \left( \underline{\mcal{L}} \right) ~|~ s \in [a,b] \}$ such that each $\psi_s\left( \underline{\mcal{L}} \right)$ is Lagrangian and contains $[\sigma_s]$ 
of $(M_s, \omega_s)$. 

Finally, letting $\mcal{L} := \underline{\mcal{L}} \times S^1$, we define
\[
	\begin{array}{cccl}
		\varphi \colon & [a, b] \times \mcal{L} & \to & \mcal{V} = \underline{\mcal{V}} \times S^1 \\
					& (s, (\underline{\ell}, t) ) & \mapsto & (\psi(s, \underline{\ell}), t)
	\end{array}, \quad (\underline{\ell}, t) \in \underline{\mcal{L}} \times S^1 = \mcal{L}.
\]
Then, the following diagram commutes: 
 \begin{equation}
		 \xymatrixcolsep{5pc}   \xymatrix{
			      [a,b] \times \mcal{L}  \ar[d]^{/S^1} \ar[r]^{\varphi} & \mcal{V} \ar[d]^{/S^1}\\
			      [a,b] \times \underline{\mcal{L}} \ar[r]^{\psi} & \underline{\mcal{V}}.
			    }
 \end{equation}
The map $\varphi$ is our desired embedding.
\end{proof}

For a given $r \in \mathrm{Im}~H$, let $L \subset (M,\omega)$ be an $S^1$-invariant Lagrangian submanifold lying on a level set $H^{-1}(r)$
and let $\sigma_p$ be the free $S^1$-orbit containing a point $p \in L$. Our goal is to compute the Maslov index of the gradient $J$-holomorphic disc $u^J_p \colon (\mathbb{D}, \pa \mathbb{D}) \to (M, L)$ in Definition~\ref{definition_gradient_holo_disk}. 
We begin with the following special case.

\begin{lemma}\label{lemma_codimtwo}
	Suppose that a gradient $J$-holomorphic disc $u_p^J$ maps the center of $\mathbb{D}$ to some $S^1$-fixed point $z_p$. 
	If the connected component $Z_p$ of $M^{S^1}$ containing $z_p$ is of codimension two in $M$, then the Maslov index of the gradient disc $u_p^J$ is two, i.e.,  $\mu([u_p^J]) = 2$.
\end{lemma}

\begin{proof}
	We first recall the following general facts about a Hamiltonian $S^1$-action :  
	\begin{itemize}
		\item (Equivariant Darboux Theorem \cite[Chap. IV-4.d]{Au}) 
		For any fixed component $Z$ of the given $S^1$-action and a point $z \in Z$, 
		there exists a system of local coordinates $z_1, \cdots, z_n$ near $z$ such that 
			\begin{itemize}
				\item $\omega = \frac{i}{2} \sum dz_i \wedge d\overline{z}_i$, and 
				\item a moment map $H : M \rightarrow \R$ is locally expressed by 
				\[
					H(z_1, \cdots, z_n) = H(z) + \sum_{i=1}^{n} c_i |z_i|^2
				\]
				for some integers $c_1, \cdots, c_n$. The integers are called {\em the weights of the $S^1$-action at $z$}.
			\end{itemize}			
		\item The set of critical points of $H$ coincides with the fixed point set of the action and is a symplectic submanifold of $(M,\omega)$. 
		In particular, the coordinates $z_{i_1}. \cdots, z_{i_k}$ corresponding to zero weights (i.e., $c_{i_j} = 0$) gives a local coordinate system near $z$ in $Z$. 
		\item (\cite[Chap. IV-2.3]{Au}) $H$ is a perfect Morse--Bott function and the Morse index equals twice the number of negative weights of the action at $z$.
		\item (\cite[Chap. IV-3.2]{Au}) Every level set of $H$ is connected. In particular, a local extremum of $H$ must be a global extremum.
	\end{itemize}
	The above statements imply that $H(Z_p) = H(z_p)$ is the maximum of $H$ in the following reason. 
	If $Z_p$ is of codimension two, it follows that all weights at $z$ except for one vanish so that $H(z_p)$ is a (local) extremum, and therefore a global extremum of $H$.  
	Also, by the construction of the gradient holomorphic disc $u_p^J$, there is a gradient flow converging to $z_p$ and this implies $z_p$ never can be a minimum.
	
	Let $c := H(Z_p) = H(z_p)$, the maximum of $H$. 
	By the equivariant Darboux theorem, there exists an open neighborhood $\mcal{U}_p$ of $z_p$ that is $S^1$-equivariantly symplectomorphic to an open neighborhood $\mcal{V}_p$ 
	of the origin in $\C^n$ 
	equipped with the standard symplectic form and the linear $S^1$-action on $\C^n$ given by 
	\[
		t \cdot (z_1, \cdots, z_n) = (t^{-1}z_1, z_2, \cdots, z_n).
	\]
Note that $(0, z_2, \cdots, z_n)$ serves as a local coordinate for $Z_p$ near $z_p$. A moment map of the action on $\C^n$ is taken as
	\[
		H(z) = c - \displaystyle \frac{1}{2} |z_1|^2, \quad z = (z_1, \cdots, z_n) \in \C^n.
	\]

	Applying Lemma \ref{lemma_flowing_lagrangian} to $\sigma_r:= \sigma_p$, $L$, and $[r, c-\epsilon]$ for any positive number $\epsilon (< (c-r))$, we obtain an $n$-dimensional manifold $\mcal{L}$ and an embedding 
	\[
		\varphi \colon [r, c-\epsilon] \times \mcal{L} \rightarrow M
	\]
	such that $L_s := \varphi(s, \mcal{L})$ is an $S^1$-invariant Lagrangian submanifold of $(M,\omega)$ lying on $H^{-1}(s)$. 
	Let $p_0 \in \mcal{L}$ such that $\varphi(r, p_0) = p$.
	
	By taking $\epsilon$ sufficiently small, we can make $\sigma_{c-\epsilon}$ contained in the Darboux neighborhood $\mcal{U}_p$. 
	Since the Lagrangian $L_r = \varphi(r, \mcal{L})$ still bounds the disc $u_p^J$ and $L_r \subset L$, the Maslov index of $u_p^J$ bounded by $L_r$ is equal to that  by $L$. 
	Moreover, it sufficies to calculate the Maslov index of $u_{\varphi(c-\epsilon, p_0)}^J$ (bounded by $L_{c-\epsilon}$) since the Maslov index is preserved through the Lagrangian isotopy by Lemma \ref{lemma_isotopy}. 	
	Therefore, we may assume without any loss of generality that $r = H(p) = c - \epsilon$.

	In the neighborhood $\mcal{V}_p$ of the origin (corresponding to the Darboux neighborhood $\mcal{U}_p$), every $S^1$-orbit containing $z = (z_1, \cdots, z_n)$ can be written as $\{ (w, z_2, \cdots, z_n) ~|~ |w| = |z_1| \}$.
	In particular, $L_{c-\epsilon} \cap\, \mcal{U}_p$ corresponds to $S^1 \times L' \subset \C \times \C^{n-1}$ in $\mcal{V}_p$ for some Lagrangian $L^\prime$ in $\C^{n-1}$ where $S^1 = \{ w \in \C ~|~ |w|^2 = 2\epsilon\}$. 
	Passing $\mcal{U}_p$ to $\mcal{V}_p$, the map $u_{\varphi(c-\epsilon, p_0)}^J$ is precisely $z \mapsto (\sqrt{2 \epsilon} \cdot z, 0, \cdots, 0)$ from $\mathbb{D}$ to $\mcal{V}_p \subset \C^n$.
	Thus, the Maslov index of $u_{c-\epsilon}^J$ is two. This finishes the proof.
\end{proof}	
		
	We list two lemmas that will be used in the proof of Theorem \ref{theorem_Maslov_index_formula}.
		
\begin{lemma}[Page 75 in \cite{MS}]\label{lemma_Chern_number}
	Consider a continuous map $u \colon S^2 \to (M, \omega)$ and suppose that $S^2$ splits into two discs $\mathbb{D}_1$ and $\mathbb{D}_2$ such that $u$ splits into two maps
	$u_+ \colon (\mathbb{D}_1, \pa \mathbb{D}_1) \to (M, L)$ and $u_- \colon (\mathbb{D}_2, \pa \mathbb{D}_2) \to (M, L)$ for some Lagrangian submanifold $L \subset (M,\omega)$. Then,
	\[
		2c_1([u]) = \mu ([u_+]) + \mu([u_-]). 
	\]
\end{lemma}

\begin{lemma}[Lemma 4.3 in \cite{AH}, Lemma 3.1 in \cite{Go}]\label{lemma_AH_Go}
	Suppose that $\pi \colon E \rightarrow S^2$ be an $S^1$-equivariant complex line bundle where the $S^1$-action on $S^2$ is given by $k$-times rotation with two fixed points $N$ and $S$. Then 
	the first Chern number $e$ of $E$ satisfies 
	\[
		-ek = m_N - m_S
	\]
	where the action near $N$ and $S$ are locally expressed as 
	\[
		t \cdot (z_b, z_f) = (t^{-k} z_b, t^{m_N} z_f) \quad \text{and} \quad t \cdot (z_b, z_f) = (t^k z_b, t^{m_S} z_f), \quad t \in S^1,
	\]
	respectively. (Here, $z_b$ and $z_f$ denote the local coordinates of the base and fibers, respectively.)
	More generally, if $\pi : E \rightarrow S^2$ is a rank $r$ $S^1$-equivariant complex vector bundle, then $E$ splits into the direct sum 
	$E_1 \oplus \cdots \oplus E_r$ of $S^1$-equivariant vector bundles and the first Chern number of $E$ is given by 
	\[
		-ek = m_N - m_S, \quad \quad m_N = m_N^1 + \cdots + m_N^r, \quad m_S = m_S^1 + \cdots m_S^r
	\]
	where $m_N^i$ and $m_S^i$ denote the weights of the $S^1$-action at $N$ and $S$ along the fiber of $E_i$, respectively.
\end{lemma}

Now, we are ready to state and prove our main theorem in this section.

\begin{theorem}\label{theorem_Maslov_index_formula}
	Let $(M,\omega)$ be a $2n$-dimensional symplectic manifold equipped with an effective Hamiltonian $S^1$-action.
	Let $L$ be an $S^1$-invariant Lagrangian submanifold of $(M,\omega)$ lying on some level set of a moment map $H$. 
	Suppose that a class $\beta \in \pi_2(M, L)$ is represented by a gradient holomorphic disc of a free $S^1$-orbit $\sigma_p$ containing a point $p \in L$. 	
	Then 
	\[
		\mu(\beta) = \mu([u_p^J]) = -2n_p
	\]	
	where $n_p$ is the sum of weights at $z_p$, the image of the center of $\mathbb{D}$ under $u_p^J$.
\end{theorem}

\begin{proof}
	Put $r = H(p)$ and choose a sufficiently small positive number $\epsilon >0$ such that $r - \epsilon$ is a regular value of $H$.
	We denote by 
	\[
		\bar{M}_{r - \epsilon} := M_{\geq r - \epsilon} / \sim
	\]
	the symplectic cut \cite{L} of $M$ along $H^{-1}(r-\epsilon)$ where $p_1 \sim p_2$ if $p_1 = t\cdot p_2$ for some $t \in S^1$.
	In general, the reduced space $\bar{M}_{r-\epsilon}$ is a symplectic orbifold
	\footnote{The notion of a {\em Hamiltonian circle action} is naturally generalized to a {\em symplectic orbifold.} 
	See \cite{LT} and \cite{Go2} for more details. For the notions {\em almost complex structures} or {\em $J$-holomorphic maps} on orbifolds, 
	we refer to Chen--Ruan's paper \cite{CR}.}
	that inherits the reduced symplectic form $\bar{\omega}_{r - \epsilon}$ from $M$. 
	Also by construction, $(\bar{M}_{r-\epsilon} \backslash M_{r - \epsilon}, \bar{\omega}_{r-\epsilon})$ and $(M_{> r- \epsilon}, \omega)$ are symplectomorphic. 
	Moreover, $\bar{M}_{r-\epsilon}$ admits the induced $S^1$-action, which is also Hamiltonian with respect to $\overline{\omega}_{r - \epsilon}$, 
	whose minimal fixed point set is diffeomorphic to $M_{r-\epsilon}$ and of codimension two, see \cite[Section 1.1]{L} for the detail. 
	Let $\bar{H}_{r-\epsilon}$ be a moment map of the induced $S^1$-action. 
	We also denote by $\bar{J}$ the induced $S^1$-invariant almost complex structure on $\bar{M}_{r-\epsilon}$ compatible with $\bar{\omega}_{r-\epsilon}$.

	Since the Lagrangian $L$ remains in $\bar{M}_{r-\epsilon}$ and nothing has changed in a small neighborhood of $L$, 
	flowing the orbit $\sigma_p$ along $\bar{J}\xi$ and $-\bar{J}\xi$, respectively, we have two gradient $\bar{J}$-holomorphic discs
	\begin{enumerate}
		\item $u^{\bar{J}}_{p,+} \colon (\mathbb{D}, \pa \mathbb{D}) \to (M, L)$ intersecting $Z_p$ at $z_p$ and 
		\item $u^{\bar{J}}_{p,-} \colon (\mathbb{D}, \pa \mathbb{D}) \to (M, L)$ intersecting $M_{r-\epsilon}$ at $[\sigma_{r-\epsilon}]$
	\end{enumerate} 
	where $\sigma_{r-\epsilon}$ denotes a free $S^1$-orbit in $H^{-1}(r-\epsilon)$ obtained by flowing $\sigma_p$ down along the negative gradient flow $-J\xi$
	and $[\sigma_{r-\epsilon}]$ denotes the corresponding point in $M_{r-\epsilon}$. Note that $[\sigma_{r-\epsilon}]$ is a smooth point in $M_{r-\epsilon}$. 
	As their boundaries match up, 
	by decomposing $S^2$ into the upper hemisphere and the lower hemisphere, 
	we may consider a map  $\widehat{u}_p^{\bar{J}} := u^{\bar{J}}_{p,+} \# u^{\bar{J}}_{p,-} \colon S^2 \to \bar{M}_{r-\epsilon}$.
	
	 Since $(\bar{M}_{r-\epsilon} \backslash M_{r - \epsilon}, \bar{\omega}_{r-\epsilon})$ and $(M_{> r- \epsilon}, \omega)$ are symplectomorphic, we have 
	 \begin{equation}\label{eq_cutba}
	 	\mu \left( [{u}_{p,+}^{\bar{J}}] \right) = \mu \left( [{u}_p^{{J}}] \right).
	 \end{equation}
	 By Lemma \ref{lemma_Chern_number}, we have
	 \begin{equation}\label{eq_decomsphanddiscs}
	 	\mu  \left( [{u}_{p,+}^{\bar{J}}] \right) + \mu \left( [{u}_{p,-}^{\bar{J}}] \right) = 2 \cdot c_1 ([\widehat{u}_p^{\bar{J}} ]). 
	 \end{equation}
	 Since the minimal fixed component is diffeomorphic to the quotient space $M_{r-\epsilon}$, see Figure \ref{figure_cut}, Lemma \ref{lemma_codimtwo} yields
	 \begin{equation}\label{eq_comdi}
	 	\mu \left( [{u}_{p,-}^{\bar{J}}] \right) = 2
	\end{equation}
	On the other hand, we can apply Lemma \ref{lemma_AH_Go} to $TM|_{\widehat{u}_p^{\bar{J}}(S^2)}$ (with $k=1$, $m_N = n_p$, and $m_S = 1$) and we get 
	\[
		\langle c_1(M), [\widehat{u}_p^{\bar{J}}] \rangle = -(m_N - m_S) = - (n_p - 1) = -n_p + 1.
	\]
	Consequently, combining~\eqref{eq_cutba},~\eqref{eq_decomsphanddiscs}, and~\eqref{eq_comdi}, we have
	\[
		\mu \left( [{u}_p^{{J}}] \right) = \mu \left( [{u}_{p,+}^{\bar{J}}] \right)  = 2 \cdot c_1 \left([\widehat{u}_p^{\bar{J}} ]\right) - \mu \left( [{u}_{p,-}^{\bar{J}}] \right) = -2n_p
	\]
\end{proof}

	\begin{figure}[h]
		\scalebox{0.9}{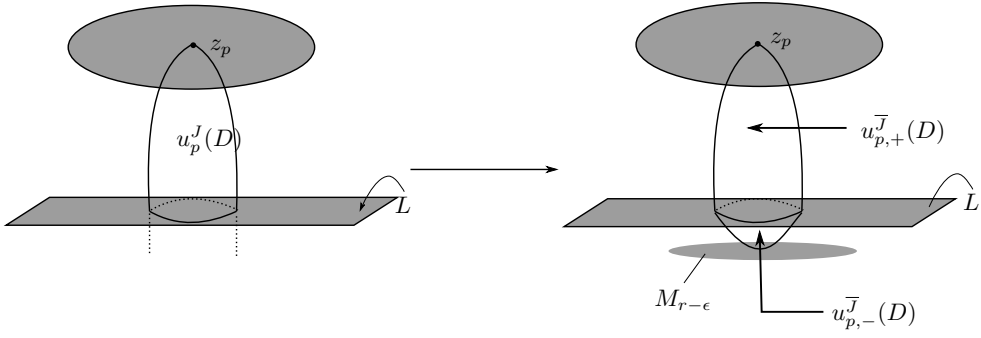}
		\caption{\label{figure_cut} Symplectic cut along $H^{-1}(r-\epsilon)$}
	\end{figure}
	\vspace{-0.2cm}
		
If the action is semifree\footnote{For a Lie group $G$ acting on a manifold $M$, we call the action is 
	{\em semifree} if the action is free on the complement of the fixed point set.}  
	near a fixed point $z \in M^{S^1}$, every non-zero weight on $T_z M$ is either $+1$ or $-1$.
In particular, twice the absolute value of the sum of 
negative (resp. positive) weights at $z$ is 
the Morse--Bott index of $H$ (resp. of $-H$) at $z$.
Therefore, we obtain the following corollary.

	\begin{corollary}\label{corollary_Maslov_index_formula}
		Let $(M,\omega)$ be a symplectic manifold equipped with an effective semifree Hamiltonian $S^1$-action with a moment map $H : M \rightarrow \R$.
		Let $u : \mathbb{D} \rightarrow M$ be a gradient holomorphic disc that contains a unique fixed point $z \in M^{S^1}$ and is 
		bounded by some $S^1$-invariant Lagrangian submanifold $L$ in a level set of $H$.
		Then the Maslov index $\mu([u])$ is the signature of $-H$ at $z$, that is, the difference between the number of negative and positive eigenvalues 
		of the Hessian of $H$ at $z$. 
		In particular, if $H$ attains the maximum at $z$, then $\mu([u])$ is the codimension of the (maximal) fixed component $Z$ containing $z$.
	\end{corollary}

In the remaining part of the section, we apply Theorem~\ref{theorem_Maslov_index_formula} and Corollary~\ref{corollary_Maslov_index_formula} to calculate the Maslov indices of some classes in some well-known examples in \cite{Cho, CO, Aur}.

\begin{example}\label{example_clifford}
Consider $\C^n$ with the standard symplectic form $\omega_0$. Let $S^1_{a}$ denote the circle centered at the origin of radius $a \in \R_+$ in $\C$. 
For $(a_1, \cdots, a_n) \in (\R_+)^n$, the torus $T^n := S^1_{a_1} \times \cdots \times S^1_{a_n}$ is a Lagrangian submanifold of $\C^n$. 
Let $\mathbb{D}$ be the unit disc in $\C$. 
It is straightforward to see that the action on $(\C^n, \omega_0)$ given by  
\[
	t \cdot (z_1, \cdots, z_k, z_{k+1}, \cdots, z_n) = (t^{-1}z_1, \cdots, t^{-1}z_k, z_{k+1},\cdots, z_n)
\]
is Hamiltonian with a moment map
\[
	\begin{array}{cccl}
		H \colon & \C^n & \rightarrow & \R \\
			& z & \mapsto & -\frac{1}{2}(|z_1|^2 + \cdots + |z_k|^2).
	\end{array}
\]

For $p = (a_1, \cdots, a_k, \cdots, a_n) \in T^n$, consider the gradient holomorphic disc of the $S^1$-orbit containing $p$ with respect to the standard complex structure, which is $S^1$-invariant. It is explicitly written as  
\[
	\begin{array}{cccl}
		u \colon & (\mathbb{D}, \pa \mathbb{D}) & \rightarrow & (\C^n, T^n) \\
		     &  z  &  \mapsto & ({a_1z, \cdots, a_k z}, {a_{k+1},\cdots,a_n}_{n-k}).
	\end{array}
\]
Note that the fixed point set
\[
	F = \{~(z_1, \cdots, z_n) \in \C^n ~|~ z_1 = \cdots = z_k = 0 \}
\]
of the action occurs at the maximum of $H$ so that there is no positive weight at $F$.
Since the action is semifree, the sum of negative weights at $F$ is exactly the Morse--Bott index of $H$ at $F$, which is the codimension of $F$.
Therefore, the Maslov index $\mu([u])$ is equal to $2(n-k)$ by Corollary \ref{corollary_Maslov_index_formula}.
According to \cite[Theorem 5.1]{CO}, the Maslov index of $[u]$ is exactly twice the intersection number between the disc and the toric divisor, which also gives us $\mu([u]) = 2(n-k)$.
\end{example}

\begin{example}\label{example_chekanov}
Let $\CP^2$ be the projective space equipped with the standard Fubini--Study form $\omega_0$ and the standard complex structure $J_0$. To generalize the Strominger--Yau--Zaslow conjecture to Fano manifolds, Auroux came up with a conic fibration in order to construct a special Lagrangian torus fibration on $\CP^2$. He explicitly exhibited the wall-crossing phenomenon using the fibration, which leads to the quantum correction for complex variables on mirror Landau--Ginzburg models, see \cite{Aur} for more details.

To recall his construction, we restrict ourselves to an affine chart of $\CP^2$ by taking $x = X / Z, y = Y / Z$ where $[X:Y:Z]$ is the homogeneous coordinate for $\CP^2$. Then, the conic fibration $f \colon \C^2 \to \C^1$ is given by $f(x,y) = xy$. We consider the fiberwise $S^1$-action determined by
\[
t \cdot (x,y) = (t^{1}x, t^{-1}y). 
\]
A periodic Hamiltonian function is chosen as
\[
\lambda(x,y) = \frac{1}{2} \cdot \frac{|x|^2 - |y|^2}{1+|x|^2 + |y|^2}
\]
with respect to $\omega_0$. For a positive number $\varepsilon > 0$, we then have a family of Lagrangian tori  
\[
T_{r, \lambda} = \{ (x,y) \in \C^2 ~|~ |xy - \varepsilon| = r,\, \lambda(x,y) = \lambda_0 \}.
\]
parametrized by $\{(r, \lambda) ~|~ r > 0, \, \lambda \in \R \}$ as in Figure~\ref{Fig_double}. 

We focus on a point $r := \varepsilon$ and $\lambda := \lambda_0 < 0$ as an example. First, choose any point $p_1 \in T_{\varepsilon, \lambda_0}$ over the origin, any point of the green circle in Figure~\ref{Fig_double}. Putting $p_1 = \left(0, y_1 \right) \in T_{r, \lambda_0}$, the gradient disc bounded by  $T_{\varepsilon, \lambda_0}$ is 
\[
	\begin{array}{cccl}
		u_{p_1}^{J_0} \colon & (\mathbb{D}, \pa \mathbb{D}) & \rightarrow & (\C^2, T_{\varepsilon, \lambda_0}) \\
		     &  z  &  \mapsto & \left(0, y_1 z\right)
	\end{array}
\]
and it contains the fixed point $u_{p_1}^{J_0}(0) = (0,0)$. Because the weights at $(0,0)$ are $1$ and $-1$, we have $n_{p_1} = 1 + (-1) = 0$ and the Maslov index of $[u_{p_1}^{J_0}]$ is zero 
by Theorem~\ref{theorem_Maslov_index_formula}.

Next, choose any point $p_2 \in T_{r, \lambda_0}$ not over the origin. For instance, let us choose a point in the red circle in Figure~\ref{Fig_double}. Putting $p_2 = (x_2, y_2)$, the gradient disc bounded by $T_{r, \lambda_0}$ is written as
\[
	\begin{array}{cccl}
		u_{p_2}^{J_0} \colon & (\mathbb{D}, \pa \mathbb{D}) & \rightarrow & (\C^2, T_{\varepsilon, \lambda_0}) \\
		     &  z  &  \mapsto & \left(x_2 z^{-1}, y_2 z\right).
	\end{array}
\]
The center of $\mathbb{D}$ maps into the point $[X:Y:Z] = [1:0:0]$ via $u_{p_2}^{J_0}$ in terms of the homogeneous coordinate. 

Note that the $S^1$ acts on $\C^2$ as a subgroup $\{(t, -t) \in T^2 ~|~ t \in S^1 \}$ and the $T^2$-action on $\C^2$ is induced from the toric action on $\CP^2$ given by 
\[
	(t_1, t_2) \cdot [X,Y,Z] = [t_1X, t_2Y, Z] \quad (t_1, t_2) \in T^2, [X,Y,Z] \in \CP^2
\]
Since the weights at $[1,0,0]$ for the toric action is $(-1, 1)$ and $(-1, 0)$ (as elements of the dual Lie algebra of $T^2$), the weights at $[1,0,0]$ with respect to the $S^1$-action are   
\[
	\langle (1, -1),(-1, 1) \rangle = -2 \quad \text{and} \quad \langle (1,-1),(-1, 0) \rangle = -1,  
\] Therefore, we have $n_{p_2} = (-2) + (-1) = -3$ and the Maslov index of the $J_0$-holomorphic disc is six by Theorem~\ref{theorem_Maslov_index_formula}. 
\end{example}

\vspace{-0.2cm}
\begin{figure}[h]
	\scalebox{0.75}{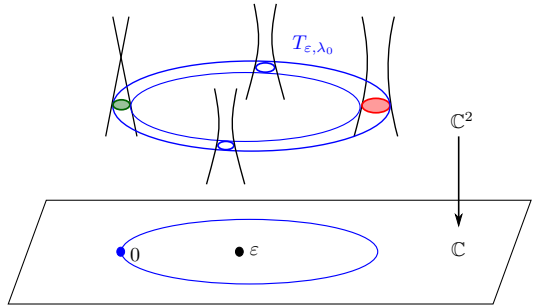}
	\bigskip
	\caption{\label{Fig_double} A family of Lagrangian tori.}	
\end{figure}
\vspace{-0.1cm}

\begin{remark}
As one can see in Example~\ref{example_chekanov}, depending on a choice of a point $p$, we can obtain gradient holomorphic discs having different Maslov indices. 
\end{remark}

\vspace{0.1cm}
\section{Gelfand--Cetiln systems}
\label{secGelfandCetilnSystems}

A {\em Gelfand--Cetlin system}, or simply a {\em GC system}, is a completely integrable system on a partial flag manifold constructed by Guillemin--Sternberg \cite{GS} as a symplectic geometric analogue of the Gelfand--Cetlin basis \cite{GC}.
As in the toric case, the image of a GC system is a convex polytope, called a {\em Gelfand--Cetlin polytope} ({\em GC polytope}).
In this section, we briefly review various notions and results on GC polytopes and GC systems in \cite{GS, CKO1} that will be needed in Section \ref{secClassificationOfMonotoneLagrangianFibers}.

\subsection{Monotone Lagrangians}
\label{ssecMonotoneLagrangians}

A symplectic manifold $(M,\omega)$ is called {\em monotone} if $[\omega] = \delta \cdot c_1(TM)$ for some positive real number $\delta$, called a 
{\em monotonicity constant}. 
A Lagrangian submanifold $L$ of $(M,\omega)$ comes with two group homomorphisms: the symplectic energy $I_\omega \colon \pi_2(M,L) \to \R$ and the Maslov index $\mu \colon \pi_2(M,L) \to \R$.
We call $L$ {\em monotone} if the symplectic area of discs bounded by $L$ is proportional to their Maslov index, that is,
\begin{equation}\label{equation_monotonicity}
	I_\omega(\beta) = c \cdot \mu(\beta) \quad \text{for every $\beta \in \pi_2(M,L)$}
\end{equation}
for some positive number $c > 0$. 
A monotone Lagrangian submanifold can exist only in a monotone symplectic manifold $(M,\omega)$. 
The monotonicity constant $\delta$ of $(M,\omega)$ is equal to $2c$ {if $(M,\omega)$ is not {\em symplectically aspherical}}\footnote{{A symplectic manifold $(M,\omega)$ is called {\em symplectically aspherical} if $\int_\alpha \omega = 0$ for every $\alpha \in \pi_2(M)$.
Note that any partial flag manifold is not symplectically aspherical.}}, see Remark 2.3 in \cite{Oh}.
In particular, we have the following. 

\begin{lemma}\label{lemma_monotonicity}
	Let $(M,\omega)$ be a monotone symplectic manifold such that $c_1(TM) = [\omega]$ and $I_\omega |_{\pi_2(M)} \neq 0$. Then a Lagrangian submanifold $L \subset (M,\omega)$ is monotone if and only if 
	\[
		2 \cdot \omega(\beta) = \mu(\beta) \quad \text{for every $\beta \in \pi_2(M,L)$.}
	\]	
\end{lemma}

\subsection{Partial flag manifolds}
\label{ssecPartialFlagManifolds}

For a positive integer $n$, consider a sequence of integers
	\begin{equation}\label{nidef}
		0 = n_0 < n_1 < n_2 < \cdots <n_r < n_{r+1} = n.
	\end{equation}
The {\em partial flag manifold} $\F(n_1, \cdots, n_r; n)$ is the space of nested sequences of complex vector subspaces of $\C^n$ defined by
	\[
		\F(n_1, \cdots, n_r; n) = \{V_{\bullet} := (0 \subset V_1 \subset \cdots \subset V_r \subset \C^n) ~|~ \dim_{\C} V_i = n_i \}.
	\]
Since the unitary group $U(n)$ acts on $\F(n_1, \cdots, n_r; n)$ transitively, it becomes a homogeneous manifold diffeomorphic to 
\[
	U(n) / (U(k_1) \times \cdots \times U(k_{r+1})), \quad k_i = n_i - n_{i-1}.
\]
One well-known fact is that any partial flag manifold is a complete Fano variety and carries a canonical projective embedding called the {\em Pl\"{u}cker embedding}.

To construct a GC system on $\F(n_1, \cdots, n_r; n)$, it is convenient to think of $\F(n_1, \cdots, n_r; n)$ as a co-adjoint orbit of $U(n)$.
Let $\frak{u}(n)$ be the Lie algebra of $U(n)$.  
Using the Killing form, the dual $\frak{u}(n)^*$ of $\frak{u}(n)$ with the co-adjoint action 
can be equivariantly identified with the set $\mcal{H}_n$ of $(n \times n)$ Hermitian matrices with the conjugate $U(n)$-action.
We choose any non-increasing sequence $\lambda := (\lambda_1, \cdots, \lambda_n)$ of real numbers such that
	\begin{equation}\label{lambdaidef}
		\lambda_1 = \cdots = \lambda_{n_1} > \lambda_{n_1 + 1} = \cdots = \lambda_{n_2} > \cdots > \lambda_{n_r +1} =
		\cdots = \lambda_{n_{r+1}} (= \lambda_n)
	\end{equation}
where $\{n_0, \cdots, n_r\}$ is given in \eqref{nidef}. 
The corresponding co-adjoint orbit is then explicitly given by 
\begin{align*}
	\mcal{O}_\lambda := \{ U \cdot \mathrm{diag}(\lambda_1, \cdots, \lambda_n) \cdot \overline{U^T} ~\colon~ U \in U(n) \} 
	= \left\{ H \in \mcal{H}_n : \mathrm{Spec}(H) = \{\lambda_1, \cdots, \lambda_n\} \right\}, 
\end{align*}
which is diffeomorphic to $U(n) / (U(k_1) \times \cdots \times U(k_{r+1})) \cong \F(n_1, \cdots, n_r; n)$. See \cite[page 51]{Au} for details. 

Any co-adjoint orbit $\mcal{O}_\lambda$ admits a $U(n)$-invariant K\"{a}hler form $\omega_\lambda$, called the {\em Kirillov--Kostant--Souriau form}.
With respect to $\omega_\lambda$,    
the above conjugate $U(n)$-action becomes Hamiltonian\footnote{For a general definition of a {\em Hamiltonian $G$-action} for any Lie group $G$, see \cite[Chap. III-1]{Au}.} 
and the inclusion map $\mu \colon \mcal{O}_\lambda \hookrightarrow \frak{u}(n)^* \cong \mcal{H}_n$ becomes a moment map for the action.
The maximal torus $T^n$ of $U(n)$, consisting of the diagonal matrices in $U(n)$, acts 
on $(\mcal{O}_\lambda, \omega_\lambda)$ in a Hamiltonian fashion with a moment map
\[
		(a_{ij}) \mapsto 
	\begin{pmatrix}
		a_{11} & 0 & \cdots & 0 \\
		0 & a_{22} & \cdots & 0 \\
		\vdots & \vdots & \ddots & \vdots \\
		0 & 0 & \cdots & a_{nn} \\
	\end{pmatrix} \mapsto (a_{11}, \cdots, a_{nn}) \in \R^n \cong \frak{t}^*.
\]
The following tells us how one should choose $\lambda$ so that $\omega_\lambda$ becomes a monotone symplectic form on $\mcal{O}_\lambda$. 

\begin{proposition}[See p.653-654 in \cite{NNU} for example]\label{proposition_monotone_lambda}
	The symplectic form $\omega_\lambda$ on $\mathcal{O}_\lambda$ satisfies
	\[
		c_1(T\mathcal{O}_\lambda) = [\omega_\lambda]
	\]
	if and only if
\[
  \lambda = (\underbrace{n-n_1, \cdots}_{k_1} ~,
  \underbrace{n-n_1-n_2, \cdots}_{k_2} ~, \cdots ~, \underbrace{n-n_{r-1}-n_r, \cdots}_{k_r} ~, \underbrace{-n_r, \cdots, -n_r}_{k_{r+1}} )
  +  (\underbrace{m, \cdots, m}_{n}),
  \label{canonical}
\]
for some $m \in \R$. 
\end{proposition}

\subsection{Gelfand--Cetlin systems}
\label{ssecGelfandCetlinSystems}

For a given $\lambda$ in \eqref{lambdaidef},
Guillemin and Sternberg \cite{GS} constructed a completely integrable system 
\begin{equation}\label{equation_GCdef}
	\Phi_\lambda := (\Phi_\lambda^{i,j})_{2 \leq i +j \leq n} \colon \mcal{O}_\lambda \rightarrow \R^{\frac{n(n-1)}{2}}
\end{equation}
on $(\mcal{O}_\lambda, \omega_\lambda)$
where $\Phi_\lambda^{i,j} \colon \mcal{O}_\lambda \rightarrow \R$ assigns the $i$-th largest eigenvalue of the leading 
principal submatrix of size $(i+j-1)$ for each element of $\mcal{O}_\lambda$. It is called a \emph{Gelfand--Cetlin system}, a {\em GC system} for short.

The image of $\Phi_\lambda$ is a convex polytope, called a {\em Gelfand--Cetlin polytope} and denoted by $\Delta_\lambda$, of dimension $\dim_\C \mcal{O}_\lambda$ and it is precisely 
the set $\{ (u_{i,j})_{2 \leq i+j \leq n} \} \subset \R^{\frac{n(n-1)}{2}}$ given by 

\begin{equation}
\begin{alignedat}{17}
  \lambda_1 &&&& \lambda_2 &&&& \lambda_3 && \cdots && \lambda_{n-1} &&&& \lambda_n  \\
  & \uge && \dge && \uge && \dge &&&&&& \uge && \dge & \\
  && u_{1,n-1} &&&& u_{2, n-2} &&&&&&&& u_{n-1, 1} && \\
  &&& \uge && \dge &&&&&&&& \dge &&& \\
  &&&& u_{1, n-2} &&&&&&&& u_{n-2, 1} &&&& \\
  &&&&& \uge &&&&&& \dge &&&&& \\
  &&&&&& \dndots &&&& \updots &&&&&& \\
  &&&&&&& \uge && \dge &&&&&&& \\
  &&&&&&&& u_{1,1} &&&&&&&&&
\end{alignedat}
\label{equation_GC-pattern}
\vspace{0.2cm}
\end{equation}
This pattern comes from the min-max principle. 
We give a list of properties of GC systems.  
\begin{theorem}\label{theorem_GC_generalfacts}
	Let $\Phi_\lambda$ be the Gelfand--Cetlin system on $(\mcal{O}_\lambda, \omega_\lambda)$ given in \eqref{equation_GCdef}.
	Let $\mathring{\mcal{O}}_\lambda := \Phi_\lambda^{-1}(\mathring{\Delta}_\lambda)$ where $\mathring{\Delta}_\lambda$ denotes the set of interior points of 
	$\Delta_\lambda$.
	\begin{enumerate}
		\item \cite[page 113]{GS} $\Phi_\lambda$ is a completely integrable system in the sense that $\Phi_\lambda$ is smooth on an open dense subset $\mcal{U} \subset \mcal{O}_\lambda$ 
		containing $\mathring{\mcal{O}}_\lambda$ satisfying the following$\colon$
		\begin{itemize}
			\item The components in $\{ \Phi_\lambda^{i,j} \}_{2 \leq i +j \leq n}$ Poisson commute with each other on $\mcal{U}$.
			\item The differentials are linearly independent in each cotangent space at $x \in \mathring{\mcal{O}}_\lambda$.
		\end{itemize}
		
		\item \cite[Proposition 5.3]{GS} The Hamiltonian vector field of each component $\Phi_\lambda^{i,j}$ generates a Hamiltonian circle action on the subset of $\mcal{O}_\lambda$ on which $\Phi_\lambda^{i,j}$ is smooth.
		\item  \cite[Theorem 5.12]{CKO1} For any $\textbf{\textup{u}} \in \Delta_\lambda$, the fiber $\Phi_\lambda^{-1}(\textbf{\textup{u}})$ is an isotropic submanifold of $(\mcal{O}_\lambda, \omega_\lambda)$.
		\item \cite[Theorem 7.9]{CKO1} For a $k$-dimensional face $F$ of $\Delta_\lambda$ and a point $\textbf{\textup{u}}$ in the relative interior $\mathring{F}$ of $F$, the fiber $\Phi_\lambda^{-1}(\textbf{\textup{u}})$ 
		is diffeomorphic to 
		$(S^1)^k \times Y_F$ for some closed manifold $Y_F$ such that 
		\[
			\pi_1(Y_F) = \pi_2(Y_F) = 0.
		\] 
		\item \cite[Corollary 7.12]{CKO1} A fiber over a point in $\mathring{\Delta}_\lambda$ is a Lagrangian torus. 
	\end{enumerate}
\end{theorem}

Now, let us investigate when components and their linear combinations of a GC system are smooth.
Since each component of $\Phi_\lambda$ is the restriction of an eigenvalue function on a leading principal submatrix contained in $\mcal{H}_m$ for some $m$ with $1 \leq m \leq n$, it is enough to study the smoothness of eigenvalue functions on $\mcal{H}_m$. 
For each $i, j \geq 1$ with $i + j -1 = m$, let
\begin{align}
	\begin{aligned}
		\mcal{U}^{i,j}_- &:= \{ A \in \mcal{H}_m ~\colon~ \text{the $i$-th largest eigenvalue is \emph{strictly} bigger than the $(i+1)$-th largest eigenvalue.}\}, \\
		\mcal{U}^{i,j}_+ &:= \{ A \in \mcal{H}_m ~\colon~ \text{the $(i-1)$-th largest eigenvalue is \emph{strictly} bigger than the $i$-th largest eigenvalue.}\}.
	\end{aligned} \label{equation_U-} 
\end{align} 
Denote by  $\Phi^{i,j} \colon \mcal{H}_m \to \R$ the function which assigns the $i$-th largest eigenvalue of each element of $\mcal{H}_m$. 
For each $(i,j)$, define a {\em partial trace} by
\begin{equation}\label{equ_partiaalttrr}
\Psi^{i,j} := \Phi^{1,m} + \cdots + \Phi^{i,j} \colon \mcal{H}_{m} \to \R. 
\end{equation}
In particualr, $\Psi^{m,1} = \Phi^{1,m} + \cdots + \Phi^{m,1}$ is a trace function and hence it is smooth on $\mcal{H}_m$. 

\begin{proposition}\label{proposition_smooth_Ham}
	For each $(i, j) \in (\Z_+)^2$ with $i + j -1 = m$, the partial trace $\Psi^{i,j}\colon \mcal{H}_{m} \to \R$ is smooth on $\mcal{U}^{i,j}_-$.	
\end{proposition}

\begin{proof}
	For any integer $k$ with $1 \leq k \leq m$, consider the $k$-th elementary symmetric polynomial 
	\[
		s_k (\Phi^{1,m}, \Phi^{2,m-1}, \cdots, \Phi^{m,1}) := \sum_{\{ i_1 < \cdots < i_k \} \subset [m]} \Phi^{i_1, m +1 - i_1} \, \cdots \, \Phi^{i_k, m +1 - i_k}, 
	\]	
	in terms of $m$ variables $\Phi^{1,m}, \cdots, \Phi^{m,1}$. 
	Note that each $s_k (\Phi^{1,m}, \Phi^{2,m-1}, \cdots, \Phi^{m,1})$ gives us  
	the $k$-th coefficient (up to sign) of the characteristic polynomial of each element in $\mcal{H}_m$ and hence it is smooth on $\mcal{H}_m$. 
	Since any symmetric polynomial in the variables $\Phi^{1,m}, \cdots, \Phi^{m,1}$ can be expressed as a polynomial in the $s_1, \cdots, s_m$ variables,
	every symmetric polynomial in the variables $\Phi^{1,m}, \cdots, \Phi^{m,1}$ is smooth on $\mcal{H}_m$.
	 
	Define 
	\[
		\begin{array}{cccl}
			{F}\colon   & \R \times \mcal{H}_m & \rightarrow & \R \\
						& (x, A) & \mapsto & \displaystyle \prod_{\{i_1 < \cdots < i_k\} \subset [m]} \left(x - ( \Phi^{i_1, m +1 - i_1}(A) +  \cdots + \Phi^{i_k, m +1 - i_k}(A) ) \right),
		\end{array}
	\]
	which is smooth because each coefficient is a symmetric polynomial in the variables $\Phi^{1,m}, \cdots, \Phi^{m,1}$.	
	 Observe that $x = \Psi^{i,j}(A) = \Phi^{1,m}(A) + \cdots + \Phi^{i,j}(A)$ is a zero of $F$. 
	 For any $A \in \mcal{U}^{i,j}_-$, $x = \Psi^{i,j}(A)$ is a simple root since it is the unique maximal root of $F$ so that $\frac{\pa F}{\pa x}$ does not vanish at $(x,A)$. 
	 By the implicit function theorem, $\Psi^{i,j}$ is smooth at $A$.  
	 Since this argument holds for every $A \in \mcal{U}_-^{i,j}$, $\Psi^{i,j}$ is smooth on $\mcal{U}_-^{i,j}$.
\end{proof}

\begin{corollary}\label{corollary_smooth}
	For each $(i, j) \in (\Z_+)^2$ with $i + j -1 = m$, $\Phi^{i,j}$ is smooth on $\mcal{U}^{i,j}_- \cap \mcal{U}^{i,j}_+$.
\end{corollary}

\begin{proof}
	A symmetric argument of the proof of Proposition~\ref{proposition_smooth_Ham} shows that $\Phi^{i,j} + \cdots + \Phi^{{m},1}$ is also smooth on $\mcal{U}_+^{i,j}$.
	Since 
	\[
		\Phi^{i,j} = \Psi^{i,j} + (\Phi^{i,j} + \cdots + \Phi^{n,1}) - (\Phi^{1, {m}} + \cdots + \Phi^{{m},1}),
	\] 
	the component $\Phi^{i,j}$ is smooth on 
	$
		\mcal{U}^{i,j} := \mcal{U}^{i,j}_+ \cap \mcal{U}^{i,j}_-.
	$
\end{proof}

	Let us restrict our attention to the co-adjoint orbit $\mcal{O}_\lambda$. 
	For each $(i,j) \in (\Z_+)^2$ with $2 \leq i+j \leq n$, we have 
	\[
		\Phi_\lambda^{i,j} := \Phi^{i,j} \circ \pi^{i,j}_{\lambda} \colon \mcal{O}_\lambda \to \mcal{H}_{i+j-1} \to \R
	\]
	where 
	$\pi^{i,j}_{\lambda} \colon \mcal{O}_\lambda \to \mcal{H}_{i+j-1}$ 
	is the projection map which assigns the leading principal submatrix of size $(i+j-1) \times (i+j-1)$ for each element in $\mcal{O}_\lambda$.
	Set a partial trace 
	\begin{equation}\label{equ_psiijpartial}
		\Psi^{i,j}_{\lambda} :=  \Phi_{\lambda}^{1,i+j-1} + \cdots + \Phi_{\lambda}^{i,j}\colon \mcal{O}_\lambda \to \R.
	\end{equation}
	Also we denote by 
	\begin{align*}
		\mcal{U}^{i,j}_{\lambda, -} := (\pi^{i,j}_\lambda)^{-1} (\mcal{U}^{i,j}_-) \cap \mcal{O}_\lambda &= \{A \in \mcal{O}_\lambda ~|~ \Phi_\lambda^{i+1,j-1}(A) < \Phi_\lambda^{i,j}(A) \}, \\
		\mcal{U}^{i,j}_{\lambda, +} := (\pi^{i,j}_\lambda)^{-1} (\mcal{U}^{i,j}_+) \cap \mcal{O}_\lambda &=  \{A \in \mcal{O}_\lambda ~|~ \Phi_\lambda^{i,j}(A) < \Phi_\lambda^{i-1,j+1}(A) \}.
	\end{align*}
	By Proposition~\ref{proposition_smooth_Ham} implies that the partial trace $\Psi^{i,j}_{\lambda}$ is smooth on $\mcal{U}^{i,j}_{\lambda, -}$.
	Also, Corollary \ref{corollary_smooth} yields that each component $\Phi_\lambda^{i,j}$ is then smooth on 
	$
		\mcal{U}_\lambda^{i,j} := \mcal{U}^{i,j}_{\lambda, +} \cap \mcal{U}^{i,j}_{\lambda, -} = \mcal{U}^{i,j} \cap \mcal{O}_\lambda,
	$
	which is open and dense in $\mcal{O}_\lambda$. 
	
	Recall that each component $\Phi_\lambda^{i,j}$ generates a Hamiltonian circle action on $\mcal{U}_\lambda^{i,j}$ by Theorem \ref{theorem_GC_generalfacts} (2). 
	Moreover, the following lemma tells us that $\Psi_{\lambda}^{i,j}$ is a moment map for a  Hamiltonian circle action with respect to the 
	symplectic form on $\mcal{U}_{\lambda, -}^{i,j}$ induced from $\omega_\lambda$.

\begin{lemma}\label{lemma_Ham_action}
For each $(i,j) \in (\Z_+)^2$ with $2 \leq i+j \leq n$, the partial trace $\Psi_{\lambda}^{i,j}$ generates a Hamiltonian circle action on 
	$(\mcal{U}_{\lambda, -}^{i,j}, \omega_\lambda |_{\mcal{U}_{\lambda, -}^{i,j}})$. 
\end{lemma}

\begin{proof}
	By Proposition \ref{proposition_smooth_Ham}, $\Psi_\lambda^{i,j}$ is smooth on $\mcal{U}_{\lambda,-}^{i,j}$ so that we only need to prove that every orbit generated by 
	the Hamiltonian vector field $\xi$ of $\Psi_\lambda^{i,j}$ is periodic on $\mcal{U}_{\lambda,-}^{i,j}$.
	Since $\Psi_{\lambda}^{i,j} = \Phi_\lambda^{1,i+j-1} + \cdots + \Phi_\lambda^{i,j}$ and each summand $\Phi_\lambda^{\ell, i+j -\ell}$ generates a Hamiltonian circle action 
	on $\mcal{U} := \bigcap_{\ell = 1}^i \mcal{U}_\lambda^{\ell, i+j - \ell}$ of $\mcal{O}_\lambda$  by \cite{GS}, so does $\Psi_{\lambda}^{i,j}$.
	Also, the set $\mcal{U}$ is also open dense in $\mcal{U}_{\lambda,-}^{i,j}$ and every orbit of the Hamiltonian vector field $\xi$ of $\Psi_\lambda^{i,j}$ is periodic on $\mcal{U}$. 
	Any orbit of $\xi$ is periodic on $\mcal{U}_{\lambda,-}^{i,j}$ since the {\em periodicity} is a closed condition. 
\end{proof}

It is not guaranteed that the Hamiltonian $S^1$-action generated by $\Psi_\lambda^{i,j}$ has an extremal fixed point set 
in $\mcal{U}_{\lambda-}^{i,j}$ since $\mcal{U}_{\lambda-}^{i,j}$ is open. Nevertheless, we will see in Section \ref{secClassificationOfMonotoneLagrangianFibers} that 
\begin{itemize}
	\item $\mcal{U}_{\lambda,-}^{i,j}$ contains the maximal fixed point set of the action generated by $\Psi_\lambda^{i,j}$, and 
	\item the action is semifree near the maximal fixed component.
\end{itemize}
that will be crucially needed for our purpose.

\begin{example}\label{example_123}
	Let $\lambda = \{2, 0, -2\}$. The co-adjoint orbit $\mcal{O}_\lambda$ is the complete flag variety 
	$
	\mcal{F}\ell(3) \simeq U(3) / \left( U(1) \right)^3.
	 $
	The corresponding GC polytope is the intersection of six half planes given by
	\[
			u_{2,1} \leq u_{1,1} \leq u_{1,2}, \quad 
			0 \leq u_{1,2} \leq 2, \quad 
			-2 \leq u_{2,1} \leq 0
	\]
	where $(u_{1,1}, u_{1,2}, u_{2,1})$ is a coordinate system of $\R^3$, see Figure \ref{figure_GC_moment_123}. In this case, we have
	$$
	\begin{cases}
	\, \mcal{U}_\lambda^{1,1} = \mcal{O}_\lambda, \\
	\, \mcal{U}_\lambda^{1,2} = \mcal{U}_\lambda^{2,1} = \{ A \in \mcal{O}_\lambda ~|~ \Phi_\lambda^{1,2}(A) > \Phi_\lambda^{2,1}(A) \}.
	\end{cases}
	$$
	$\Phi_{1,2}$ and $\Phi_{2,1}$ are \emph{not} smooth on $\Phi_\lambda^{-1}(0,0,0)$. But, their sum $\Phi_\lambda^{1,2} + \Phi_\lambda^{2,1}$ is smooth on $\mcal{O}_\lambda$.
	\begin{figure}[h]
		\scalebox{1.35}{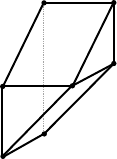}
		\caption{\label{figure_GC_moment_123} Gelfand--Cetlin polytope $\Delta_\lambda$ for $\lambda = \{2, 0, -2\}$}
	\end{figure}
\end{example}

\subsection{Ladder diagrams}
\label{ssecLadderDiagrams}
Let $\lambda$ be given in \eqref{lambdaidef}.
According to the work \cite{ACK} of the first author with An and Kim, the face structure of a GC polytope $\Delta_\lambda$ can be understood in terms of certain subgraphs of a {\em ladder diagram}. 

\begin{definition}
	Consider $\Gamma_{\Z^2} \subset \R^2$, the square grid graph such that
	\begin{itemize}
		\item its vertex set is $\Z^2 \subset \R^2$ and
		\item each vertex $(a,b) \in \Z^2$ connects to exactly four vertices $(a, b \pm 1)$ and $(a \pm 1, b)$.
	\end{itemize}
	The {\em ladder diagram} $\Gamma_\lambda$ is the induced subgraph whose vertex set is 
	\[
		\bigcup_{j=0}^r \,\, \left\{ (i,j) \in \Z^2 ~\big{|}~ \, (i,j) \in [n_j, n_{j+1}] \times [0,n-n_{j+1}] \right\}.
	\]
	Here, we denote the ladder diagram by $\Gamma_\lambda$ because $\lambda$ determines $(n_1, \cdots, n_r; n_{r+1})$ in~\eqref{nidef}.
\end{definition}	

We call the {\em origin} the vertex $O = (0,0)$ and a {\em top vertex} a farthest vertex from the origin with respect to the taxi-cap metric. In other words, $(i,j)$ is a top vertex of $\Gamma_\lambda$ 
if $(i,j)$ is a vertex of $\Gamma_\lambda$ and $i+j = n$.
(See Figure \ref{figure_ld123}, the blue dots denote the top vertices.)
Also, a shortest path from the origin to a top vertex of $\Gamma_\lambda$ is called a {\em positive path}. 

\begin{example}
Consider two examples$\colon \lambda = (2, 0, -2)$ and $\lambda^\prime = (3, 3, 0, -3, -3)$. Since the corresponding tuples $(n_1, \cdots, n_r; n_{r+1})$ in~\eqref{nidef} are 
$(1,2;3)$ and $(2,3;5)$. The ladder diagrams $\Gamma_\lambda$ and $\Gamma_{\lambda^\prime}$ are in Figure~\ref{figure_ld123}. The top vertices of $\Gamma_\lambda$ and $\Gamma_{\lambda^\prime}$ are $\{(1,2), (2,1)\}$ and $\{(2,3), (3,2)\}$ respectively.
\begin{figure}[h]
	\scalebox{0.47}{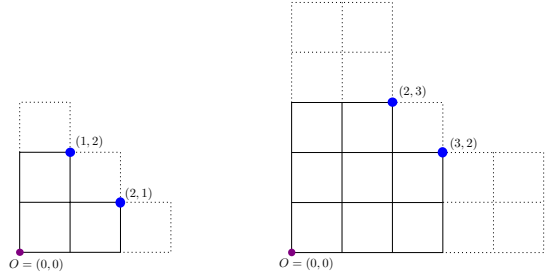}
	\caption{\label{figure_ld123} Ladder diagrams $\Gamma_\lambda$ and  $\Gamma_{\lambda^\prime}$.}
\end{figure}
\end{example}
	
Now, we equip a face structure on $\Gamma_\lambda$ as follows.

\begin{definition}[Definition 1.5 in \cite{ACK}]\label{def_face}
	A subgraph $\gamma$ of $\Gamma_\lambda$ is called a {\em face} of $\Gamma_\lambda$ if 
	\begin{itemize}
		\item $\gamma$ contains all top vertices of $\Gamma_\lambda$, and 
		\item $\gamma$ can be represented as a union of positive paths.
	\end{itemize}
	The {\em dimension} of a face $\gamma$ of $\Gamma_\lambda$ is defined by 
	$$\dim \gamma := \mathrm{rank}_\Z H_1(\gamma; \Z),$$
	that is, $\dim \gamma$ is the number of bounded regions in $\gamma$.
\end{definition}
	
The first author together with An and Kim proved the following correspondence.
	
\begin{theorem}[Theorem 1.9 in \cite{ACK}]\label{theorem_ACK}
	There exists a bijective map
		\begin{equation}\label{equ_orderpreservingmap}
			\{~\text{faces of}~\Gamma_\lambda \} \stackrel{\Psi} \longrightarrow \{~\text{faces of} ~\Delta_\lambda\}
		\end{equation}
	such that for any faces $\gamma$ and $\gamma'$ of $\Gamma_\lambda$,
	\begin{itemize}
		\item $\dim \Psi(\gamma) = \dim \gamma$ 
		\item $\gamma \subset \gamma'$ if and only if $\Psi(\gamma) \subset \Psi(\gamma')$.
	\end{itemize}
\end{theorem}

We briefly describe the order-preserving bijective map $\Psi$ in~\eqref{equ_orderpreservingmap}. 
Denote the set of the double indices of the GC system $\Phi_\lambda$ by $I_\lambda \subset (\Z_+)^2$ 
such that $(i,j) \in I_\lambda$ if and only if $\Phi_\lambda^{i,j}$ is a non-constant function on $\mcal{O}_\lambda$.
For each $(i,j) \in I_\lambda$, let $\square^{(i,j)}$ be the closed region bounded by the unit square in $\R^2$ 
such that its vertices are lying on the lattice $\Z^2$ and the upper-right vertex is located at $(i,j)$. 
The non-constant GC components $\{ \Phi_\lambda^{i,j} \}_{(i,j) \in I_\lambda}$ and the unit boxes $\{ \square^{(i,j)} \}$ in the ladder diagram $\Gamma_\lambda$ are then in one-to-one correspondence because of the GC pattern in~\eqref{equation_GC-pattern}.
Each point in $\Delta_\lambda$ corresponds to
a ``{\em filling of each $\square^{(i,j)}$ by a real number}'' obeying the inequalities in \eqref{equation_GC-pattern}. 
For each face $\gamma$ in $\Gamma_\lambda$, we take the intersection of all facets defined by $u_{i+1,j} = u_{i+1, j+1}$ (resp. $u_{i,j+1} = u_{i+1, j+1}$) if $\gamma$ does \emph{not} contain the line segment $\overline{(i,j)(i+1,j)}$ (resp. $\overline{(i,j)(i,j+1)}$). The face supported by the intersection is then the image $\Psi(\gamma)$. 

\begin{example}
Let us revisit Example~\ref{example_123} when $\lambda = (2, 0, -2)$. 
The first graph in Figure~\ref{figure_twofaces} is the ladder diagram $\Gamma_\lambda$.
As an example, let us consider four subgraphs $\gamma^{\vphantom{\prime}}_1, \gamma^{\vphantom{\prime}}_2, \gamma_{\vphantom{1}}^\prime$, and $\gamma_{\vphantom{1}}^{\prime\prime}$ of $\Gamma_\lambda$ given in Figure~\ref{figure_twofaces}. Observe that $\gamma_1$ and $\gamma_2$ are faces and $\gamma^\prime$ and $\gamma^{\prime\prime}$ are \emph{not} faces. The faces $\gamma_1$ and $\gamma_2$ correspond to the faces contained in $u_{1,2} = 0$ and $u_{1,2} = 0, u_{1,1} = u_{2,1}$ respectively. Since $\gamma_1$ (resp. $\gamma_2$) has two (resp. one) bounded regions, the dimension of the face $\Psi(\gamma_1)$ (resp. $\Psi(\gamma_2)$) is two (resp. one). Because $\gamma_1$ contains $\gamma_2$, the face $\Psi(\gamma_1)$ contains the face $\Psi(\gamma_2)$. 

\begin{figure}[h]
	\scalebox{1}{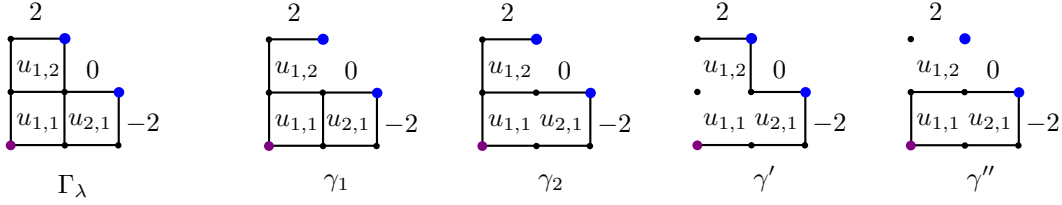}
		\caption{\label{figure_twofaces} Faces and non-faces of $\Gamma_\lambda$ where $\lambda =(2, 0, -2)$}
\end{figure}

Note that the front triangle of $\Delta_\lambda$ in Figure~\ref{figure_frontface} is the image of $\gamma_1$ in Figure~\ref{figure_twofaces} via $\Psi$.
It has three vertices$\colon$
$$
v_1 := (u_{1,1} = 0, u_{1,2} = 0, u_{2,1} = 0), v_2 := (0, 0, -2), v_3 = (-2, 0, -2).
$$
We denote by $\gamma_{v_i}$ the inverse image of $v_i$ via $\Psi$. 
Let $e_{ij}$ be the edge connecting $v_i$ and $v_j$ where $1 \leq i < j \leq 3$ and $\gamma_{e_{ij}}$ the corresponding face to $e_{ij}$ in $\Gamma_\lambda$.
The corresponding faces contained in $\gamma_1$ of $\Gamma_\lambda$ are illustrated in Figure~\ref{figure_frontface}. 

\begin{figure}[h]
	\scalebox{1.1}{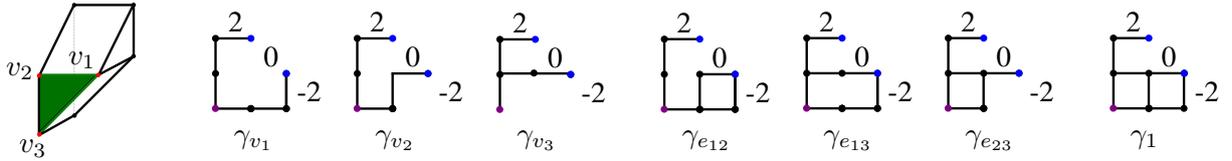}
		\caption{\label{figure_frontface} Faces contained in the front face of $\Delta_\lambda$.}
\end{figure}
\end{example}

\subsection{The topology of Gelfand--Cetlin fibers}
\label{ssecTopologyOfGelfandCetlinFibers}

Let $\lambda$ be given in \eqref{lambdaidef} and let $f$ be any face of the GC polytope $\Delta_\lambda$ of dimension $k$. 
For any point $\textbf{\textup{u}}$ in its relative interior $\mathring{f}$, the fiber $\Phi_\lambda^{-1}(\textbf{\textup{u}})$ is an isotropic submanifold of $(\mcal{O}_\lambda, \omega_\lambda)$ by Theorem \ref{theorem_GC_generalfacts}.
In \cite[Theorem 5.12]{CKO1}, the authors and Oh developed a way of ``reading off'' the topology of $\Phi^{-1}_\lambda(\textbf{\textup{u}})$ from the face $\gamma_f$ corresponding to $f$ in $\Gamma_\lambda$.
In this section, we review the algorithm and compute some homotopy groups of fibers of $\Phi_\lambda$.

We first consider a closed region in $\R^2$ bounded by $\Gamma_\lambda$, called the {\em board}.
Then we think of the set of edges in $\gamma_f$ as {\em walls}. We still denote the ``board with the walls'' by $\gamma_f$ if there is no danger of 
confusion. See Figure \ref{figure_board_walls} for example.

\begin{figure}[h]
	\scalebox{0.85}{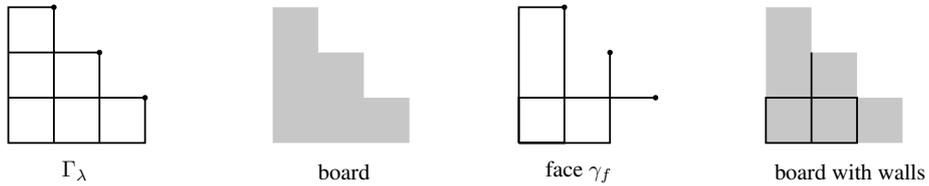}
	\caption{\label{figure_board_walls} Board with walls}
\end{figure}

Now, we define a notion ``{\em filling of $\gamma_f$ with $L$-blocks}'' which captures the topology of a fiber over a point contained in the relative interior of $f$.

\begin{definition}[\cite{CKO1}, Definition 5.16]
	Let $\square^{(i,j)}$ be defined in Section \ref{ssecLadderDiagrams}. 
	For each positive integer $k \in \Z_{\geq 1}$, an {\em $L_k$-block at $(a,b)$} is defined by
	\[
		L_k(a,b) := \bigcup_{\substack{0 \leq p \leq k-1}} \square^{(a,b+p)} \,\, \cup \bigcup_{\substack{0 \leq p \leq k-1}} \square^{(a+p,b)}.
	\]
	An \emph{$L$-block} is meant to be a $L_k(a,b)$-block for some $k \geq 1$ and some $(a,b) \in (\Z_+)^2$. 
\end{definition}

	\vspace{-0.6cm}
	\begin{figure}[h]
	\scalebox{0.85}{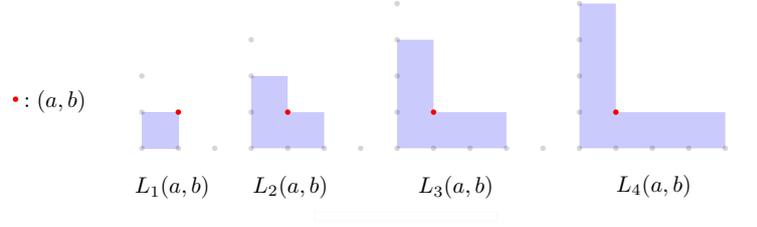}
	\vspace{-0.3cm}
	\caption{\label{figure_L_shape} $L$-blocks}
	\end{figure}

Regarding a face $\gamma_f$ of $\Gamma_\lambda$ as a board with walls, we will fill $\gamma_f$ with $L$-blocks satisfying 
\begin{align}
	\begin{aligned}
		\text{(1) The interior of a $L$-block $L_k(a,b)$ does \emph{not} contain a wall of $\gamma_f$, and} \hspace{3.07cm}\\
		\text{(2) Both the rightmost edge and the top edge of $L_k(a,b)$ are walls of $\gamma_f$.} \hspace{3cm}
	\end{aligned}\label{equation_L_block}
\end{align}

\begin{definition}
	A {\em filling of $\gamma_f$ with $L$-blocks} is defined as a collection of all $L$-blocks satisfying the conditions \eqref{equation_L_block}. 
	We say that $\gamma_f$ is {\em fillable by $L$-blocks} if the filling of $\gamma_f$ covers the whole board $\gamma_f$.
	If $\gamma_f$ is fillable by $L$-blocks, then we call a face $\gamma_f$ (resp. $f$) a {\em Lagrangian face} of $\Gamma_\lambda$ (resp. $\Delta_\lambda$).
\end{definition}

The following theorem characterizes Lagrangian fibers.

\begin{theorem}[\cite{CKO1}, Corollary 5.23]\label{theorem_Lblockfillablemain}
	Let $f$ be a face of $\Delta_\lambda$ and let $\textbf{\textup{u}}$ be any point in its relative interior.
	The fiber $\Phi_\lambda^{-1}(\textbf{\textup{u}})$ is Lagrangian if and only if $\gamma_f$ is fillable by $L$-blocks.
\end{theorem}

\begin{remark}\label{remark_sym}
The face $\gamma_f$ corresponding to $f$ cuts $\Gamma_\lambda$ into regions. 
By the min-max principle, there exists a \emph{unique} point in each region that is closest from the origin. We say such a point the \emph{bottom-left} vertex of a region.
If $f$ is Lagrangian, $\gamma_f$ is $L$-block fillable and hence each cut region is symmetric with respect to the line with slope $1$ passing through its bottom-left vertex. 
\end{remark}

\begin{example}\label{example_L_fill}
	Let $\lambda = \{6,4,2,0,-2,-4,-6\}$. Then the co-adjoint oribt $\mcal{O}_\lambda$ is diffeomorphic to a complete flag variety. Consider two faces $\gamma_f$ and 
	$\gamma_g$ of $\Gamma_\lambda$ and their fillings with $L$-blocks as in Figure~\ref{figure_fillingwithL}. Since the filling of $\gamma_f$ covers whole $\gamma_f$, 
	any fiber over a point in the relative interior of $f$ is Lagrangian. On the other hand, one cannot fill the regions $\square^{(1,5)}$ and $\square^{(5,1)}$ in $\gamma_g$
	(white area in the most right one in Figure \ref{figure_fillingwithL}) using $L_1$-blocks since the second condition of \eqref{equation_L_block} is violated.
	Thus $\gamma_g$ is \emph{not} fillable by $L$-blocks and any fiber over a point in $g$ is not Lagrangian. 
	\begin{figure}[h] 
		\scalebox{0.9}{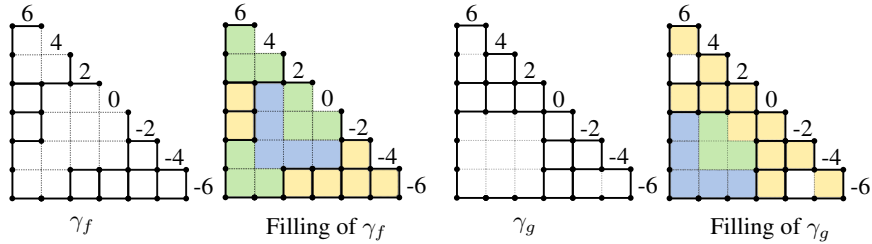}
		\caption{\label{figure_fillingwithL} Filling with $L$-blocks}
	\end{figure}
\end{example}

Now we explain the algorithm provided in \cite[Theorem 5.12]{CKO1} that decodes the topology of 
$\Phi_\lambda^{-1}(\textbf{\textup{u}})$ from the fillings of $\gamma_f$ by $L$-blocks. 

 \begin{itemize} 
\item {\bf Step 1.} For each $L_k$-block located at $(i,j) \in (\Z_+)^2$, we associate $S_{(i,j)}^{2k-1}$, the $(2k-1)$-dimensional sphere labeled by $(i,j)$. 
\item {\bf Step 2.} For each integer $\ell > 1$, we denote by 
	\[
		F_\ell := \prod_{i+j+k = \ell+1} S_{(i,j)}^{2k-1},
	\]
	which is a product of odd spheres.
\item {\bf Step 3.} Then $\Phi_\lambda^{-1}(\textbf{\textup{u}})$ is the total space of an iterated bundle 
	\[
		E_n \rightarrow E_{n-1} \rightarrow \cdots \rightarrow E_2 \rightarrow E_1 = \mathrm{point}
	\]
	where $E_\ell$ is an $F_\ell$-bundle over $E_{\ell-1}$ for every $\ell = 2,\cdots, n$. \vspace{0.5cm} 
\end{itemize}
Moreover, \cite[Theorem 7.9]{CKO1} tells us that the filling of $\gamma_f$ has exactly $(\dim f)$-number of $L_1$-blocks and 
every $S^1$-factor corresponding to each $L_1$-block is a trivial factor so that 
\begin{equation}\label{equ_circlefacctt}
	\Phi_\lambda^{-1}(\textbf{\textup{u}}) \cong (S^1)^{\dim f} \times Y_f
\end{equation}
where $Y_f$ is the total space of an iterated bundle of products of odd spheres. 
As a corollary, we obtain the following.

\begin{corollary}\label{prposition_torifactors}
Let $\gamma_f$ be the corresponding face of a face $f$ in $\Delta_\lambda$. 
\begin{itemize}
\item If the filling of $\gamma_f$ with $L$-blocks does \emph{not} contain any $L$-blocks, then $f$ is a vertex and moreover the fiber over the vertex is a point. 
\item If the filling of $\gamma_f$ with $L$-blocks does \emph{not} contain any $L_k$-blocks for $k \geq 2$, then the fiber over any point in the relative interior of $f$ is a torus of dimension $\dim f$.
\end{itemize}
\end{corollary}

Consequently, we have the following. 

\begin{proposition}[Proposition 7.11 in \cite{CKO1}]\label{theorem_homotopygroup}
	Let $f$ be a $k$-dimensional face of $\Delta_\lambda$ and $\textbf{\textup{u}} \in \mathring{f}$ be a point in its relative interior. Then the following hold. 
	\begin{enumerate}
		\item $\pi_1(\Phi_\lambda^{-1}(\textbf{\textup{u}})) = \Z^k$. 
		\item $\pi_2(\Phi_\lambda^{-1}(\textbf{\textup{u}})) = 0$. 
		\item $\pi_2(\mcal{O}_\lambda, \Phi_\lambda^{-1}(\textbf{\textup{u}})) = \pi_2(\mcal{O}_\lambda) \oplus \pi_1(\Phi_\lambda^{-1}(\textbf{\textup{u}})) = \pi_2(\mcal{O}_\lambda) \oplus \Z^k$. 
	\end{enumerate}
\end{proposition}

\vspace{0.2cm}
\section{Monotone Lagrangian Gelfand--Cetlin fibers}
\label{secClassificationOfMonotoneLagrangianFibers}

In this section, using the Maslov index formula derived in Section~\ref{secMaslovIndicesAndChernNumbers}, we classify all monotone Lagrangian fibers of $\Phi_\lambda$ for a monotone partial flag manifold $(\mcal{O}_\lambda, \omega_\lambda)$. 

By scaling and translating $\lambda$, without loss of generality, we may assume
\begin{equation}\label{equation_lambda_monotone}
  \lambda = (\underbrace{n-n_1, \cdots}_{k_1} ~,
  \underbrace{n-n_1-n_2, \cdots}_{k_2} ~, \cdots ~, \underbrace{n-n_{r-1}-n_r, \cdots}_{k_r} ~, \underbrace{-n_r, \cdots, -n_r}_{k_{r+1}} )
\end{equation}
so that $c_1(\mcal{O}_\lambda) = [\omega_\lambda]$ by Proposition \ref{proposition_monotone_lambda} with a choice of $m = 0$. 
The GC polytope $\Delta_\lambda$ 
is then a reflexive\footnote{A convex polytope $\Delta$ is called {\em reflexive} if $\Delta^* = \Delta$.} polytope, 
which has a unique interior integral lattice point $\textbf{\textup{u}}_{\Delta_\lambda} \in \mathring{\Delta}_\lambda$ 
such that the affine distances from $\textbf{\textup{u}}_{\Delta_\lambda}$ to each facet are all equal. 
In this sense, we call $\textbf{\textup{u}}_{\Delta_\lambda}$ the {\em center of $\Delta_\lambda$}. 
Choosing $\lambda$ in~\eqref{equation_lambda_monotone}, the center is explicitly written as 
\begin{equation}\label{equ_center_poly}
	\textbf{\textup{u}}_{\Delta_\lambda} = (u_{i,j} := j - i) \in \Delta_\lambda.
\end{equation}

Let $f$ be a Lagrangian face of $\Delta_\lambda$ of dimension $k$ for some $k \geq 0$ and 
$\gamma_f$ the corresponding face in the ladder diagram $\Gamma_\lambda$. 
Recall that the dimension of the face $\gamma_f$ coincides with the number of bounded regions in $\gamma_f$ by Theorem~\ref{theorem_ACK}. Note that there exists a unique rightmost vertex, a vertex farthest from the origin, of each bounded region by the min-max principle.
We label the upper rightmost vertex of each bounded region of $\gamma_f$ by $\{(a_1,b_1), \cdots, (a_k, b_k) \}$ respecting the following order
\begin{equation}\label{equation_labeling}
	\text{$i < j \,$ if and only if either} \quad  \begin{cases} \text{$a_i + b_i = a_j + b_j$ and $a_i < a_j$, or } \\ a_i + b_i > a_j + b_j. \end{cases}
\end{equation}
We denote the set of such double indices by
\begin{equation}\label{equ_doubleindices}
\mcal{I}_f := \{(a_1,b_1), \cdots, (a_k, b_k) \}.
\end{equation}

Under the choice of $\lambda$ as in~\eqref{equation_lambda_monotone}, the {\em center} of a face $f$ is a point $\textbf{\textup{u}}_f = (u_{i,j})$ defined by 
	\begin{equation}\label{equ_aibi}
		u_{a_i, b_i} = b_i - a_i, \quad \mbox{for all } (a_i, b_i) \in \mcal{I}_f
	\end{equation}	
	 Note that the positive paths in the face $\gamma_f$ cut the ladder diagram $\Gamma_\lambda$ into several regions. All components in any bounded region have same values
	(determined by~$u_{a_i, b_i}$'s) and all components in any unbounded regions are determined by the sequence $\lambda$ in~\eqref{equation_lambda_monotone}. 
	Therefore,~\eqref{equ_aibi} locates a unique point $\textbf{\textup{u}}_f$ in the face $f$.

\begin{example}
Choose $\lambda = (5,5,2,0,-2,-4,-6)$ as an example and consider the face $\gamma_f$ in Figure~\ref{figure_rightmost}.
It has eight bounded regions and their upper rightmost vertices are ordered as follows$\colon$ 
$$
\begin{cases}
(a_1, b_1) = (2, 5), (a_2, b_2) = (3,4), (a_3, b_3) = (6,1), (a_4, b_4) = (3,3),\\ 
(a_5, b_5) = (5,1), (a_6, b_6) = (4,1), (a_7, b_7) = (1,3), (a_8, b_8) = (3,1).
\end{cases}
$$
The center of $\gamma_f$ is given as in the second figure in Figure~\ref{figure_rightmost}.
\end{example}

	\begin{figure}[H]
		\scalebox{0.65}{
		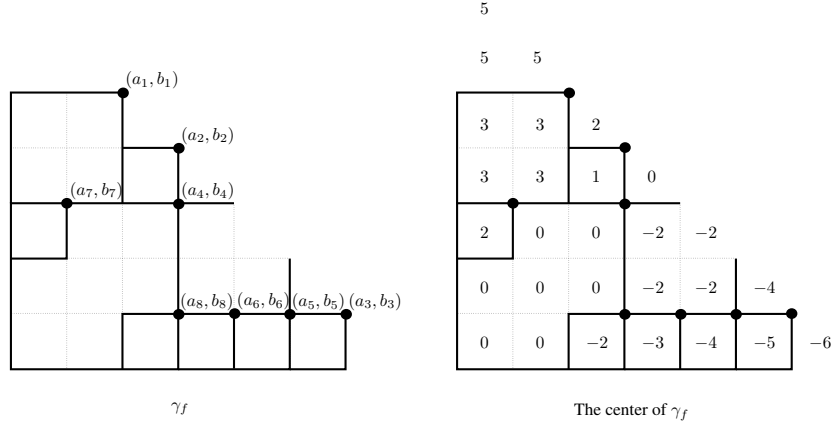}
		\caption{\label{figure_rightmost} The ordering of the upper rightmost vertices of the bounded regions and the center}
	\end{figure}

We are now ready to state the main theorem of this section. 

\begin{theorem}[Classification of monotone Lagrangian fibers]\label{theorem_main}
	Let $f$ be a Lagrangian face of the Gelfand--Cetlin polytope $\Delta_\lambda$ for $\lambda$ in~\eqref{equation_lambda_monotone}. For $\textbf{\textup{u}} \in \mathring{f}$, the fiber
	$\Phi_\lambda^{-1}(\textbf{\textup{u}})$ is monotone Lagrangian if and only if $\textbf{\textup{u}}$ is the center of $f$. 
\end{theorem}

\begin{example}
Let $\lambda = (3,3,3, -3,-3,-3)$. The co-adjoint orbit $\mcal{O}_\lambda$ is $\mathrm{Gr}(3,\C^6)$. 
Because of Theorem~\ref{theorem_Lblockfillablemain}, by classifying the $L$-block fillable faces, one can classify all Lagrangian faces of $\Delta_\lambda$. In this case, we have seven Lagrangian faces. 
Theorem~\ref{theorem_main} says that the fiber at the center of each Lagrangian face is monotone. Moreover, they are all monotone Lagrangian GC fibers and listed in Figure~\ref{figure_monotoneGr(3,6)}. It admits one monotone $T^9$-fiber, one monotone $U(3)$-fiber, one monotone $(S^3)^2 \times T^3$-fiber, and four monotone $S^3 \times T^6$-fibers. 
	\begin{figure}[H]
		\scalebox{0.6}{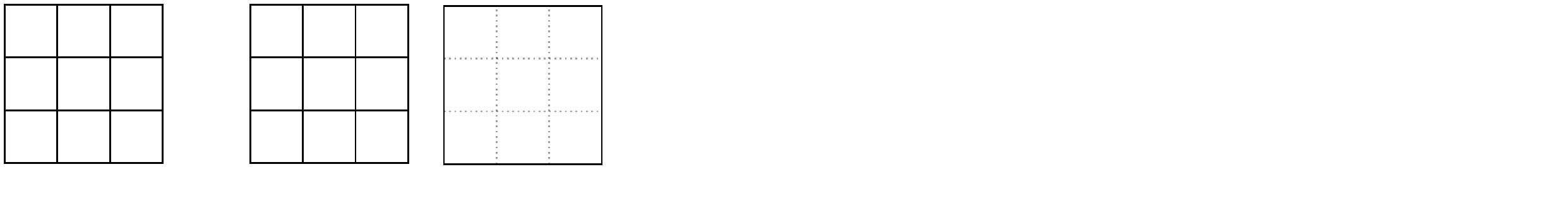}
		\caption{\label{figure_monotoneGr(3,6)} Monotone Lagrangian fibers  in $\mathrm{Gr}(3,\C^6)$}
	\end{figure}
\end{example}

\begin{example}
Let $\lambda = (4, 2, 0, -2, -4)$. The co-adjoint orbit $\mcal{O}_\lambda$ is then a complete flag manifold $\mcal{F}\ell(5)$. 
The following figure \ref{figure_F5} illustrates all monotone Lagrangian GC fibers in $\mcal{F}\ell(5)$.
\begin{figure}[H]
	\scalebox{0.6}{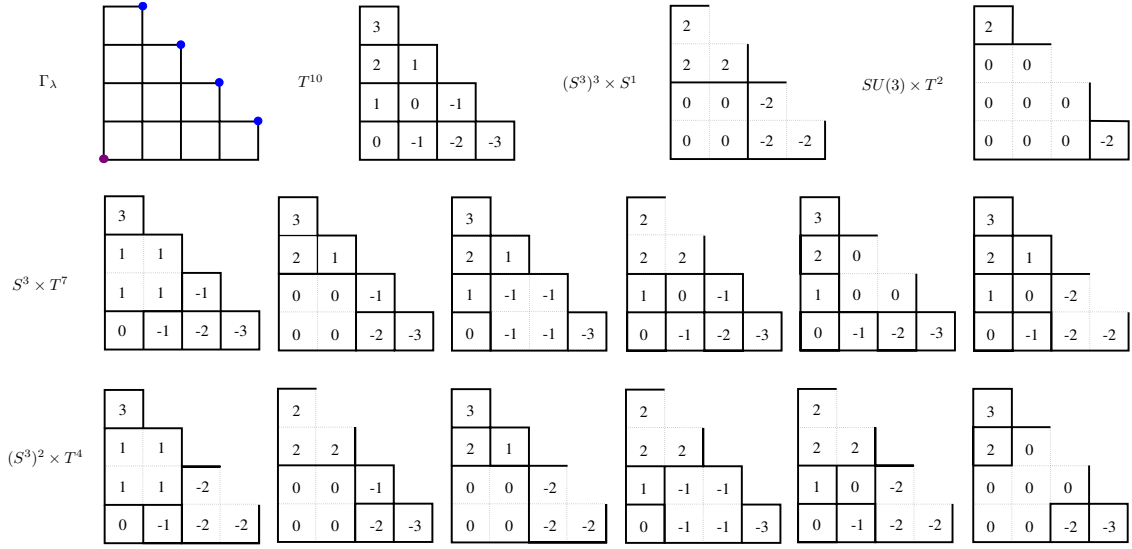}
	\caption{\label{figure_F5} Monotone Lagrangian fibers in $\mcal{F}\ell(5)$}
\end{figure}
\end{example}

The remaining part of this section is reserved for proving Theorem~\ref{theorem_main}. 
We begin by characterizing the center of a Lagrangian face in terms of partial traces $\Psi^{a,b}_\lambda$'s in~\eqref{equ_psiijpartial}.
Notice that the partial trace $\Psi^{a,b}_\lambda$ factors through $\R^{\dim \Delta_\lambda}$. That is,
we have the following commutative diagram
	 \begin{equation}\label{equ_factorthru}
		    \xymatrix{
			      \mcal{O}_\lambda  \ar[dr]_{\Phi_\lambda} \ar[rr]^{\Psi^{a,b}_\lambda}
			      & & \R  \\ 
			      & \R^{\dim \Delta_\lambda} \ar[ur]_{\overline{\Psi}^{a,b}_\lambda} &
			    }
	  \end{equation}
where $\overline{\Psi}^{a,b}_\lambda \colon \R^{\dim \Delta_\lambda} \to \R$ is given by $u_{1, a+b-1} + u_{2, a+b-2} + \cdots + u_{a,b}$. By abuse of notation, we denote $\overline{\Psi}^{a,b}_\lambda$ by ${\Psi}^{a,b}_\lambda$ from now on. In a similar vein, the projection $\overline{\Phi}^{a,b}_\lambda \colon \R^{\dim \Delta_\lambda} \to \R$ given by $\textbf{\textup{u}} \mapsto {{u}}_{a,b}$ is denoted by ${\Phi}^{a,b}_\lambda$

\begin{lemma}\label{lemma_charpartialtraces}
Let $f$ be a Lagrangian face of $\Delta_\lambda$ when $\lambda$ is given in~\eqref{equation_lambda_monotone}. 
A point $\textbf{\textup{u}}$ in the face ${f}$ is the center of $f$ if and only if
	\begin{equation}\label{equ_traces}
		\Psi^{a,b}_\lambda(\textbf{\textup{u}}_{\Delta_\lambda}) = \Psi^{a,b}_\lambda (\textbf{\textup{u}})
	\end{equation}
	for any $(a, b) \in \mcal{I}_f$ where $\mcal{I}_f$ is in~\eqref{equ_doubleindices}.
\end{lemma}

\begin{proof}
Suppose that $\textbf{\textup{u}}$ is the center of a Lagrangian face $f$.
We have to confirm that~\eqref{equ_traces} holds for each $(a,b) \in \mcal{I}_f$.
We cut the region bounded by $\Gamma_\lambda$ along the face $\gamma_f$ into several regions.
By the min-max principle, there exists a \emph{unique} point in each cut region that is closest from the origin. We say such a point the \emph{bottom-left} vertex of a cut region.
Since $f$ is Lagrangian, $\gamma_f$ is $L$-block fillable and hence each cut region is symmetric with respect to the line with slope one passing through its bottom-left vertex. 
For any $(a, b) \in \mcal{I}_f$, consider the following set of double indices
\[
	{\bf LU}_{(a,b)}\footnote{${\bf LU}$ stands for ``left-upper''.} := \left\{ {(j, a + b - j)} ~|~ j =1, \cdots, a \right\}
\] 
in the anti-diagonal given by $x + y = a + b$. The set ${\bf LU}_{(a,b)}$ is grouped into $s$ disjoint subsets 
\[
	{\bf LU}_{(a,b)} = \bigcup_{i=1}^s  \left\{ {(j, a + b - j)} ~|~ j =r_{i-1} + 1, \cdots, r_i \right\} \quad \text{for some} \quad r_0 := 0 < r_1 < \cdots < r_s := a
\]
such that each subset is contained in exactly one cut region. 
By symmetry of cut regions, we have
\[
	\begin{array}{ccl}\vs{0.1cm}
		\Phi^{r_{i-1} + 1,a + b - r_{i-1} - 1}_\lambda (\textbf{\textup{u}}) = \cdots = \Phi^{r_{i}, a + b - r_{i}}_\lambda (\textbf{\textup{u}}) & = & 
		\ds \frac{(a + b - 2r_{i-1} - 2)+ (a + b - 2 r_{i})}{2} 
		\\ \vs{0.1cm}
		& =& a + b - r_{i-1} - r_i - 1.
	\end{array}
\]
Thus,
\begin{equation}\label{equ_psiabri}
\Psi^{r_i, a+b-r_i}_\lambda (\textbf{\textup{u}}) - \Psi^{r_{i-1}, a+b-r_{i-1}}_\lambda (\textbf{\textup{u}}) = \sum_{j=1}^{r_i - r_{i-1}} \Phi^{r_{i-1} + j,a + b - r_{i-1} - j}_\lambda (\textbf{\textup{u}}) = (r_i - r_{i-1}) (a + b - r_{i-1} - r_i - 1).
\end{equation}
Since $\Phi^{i,j}_\lambda (\textbf{\textup{u}}_{\Delta_\lambda}) = j - i$ by~\eqref{equ_center_poly},
we see
\begin{equation}\label{equ_psiabri2}
\Psi^{r_i, a+b-r_i}_\lambda (\textbf{\textup{u}}_{\Delta_\lambda}) - \Psi^{r_{i-1}, a+b-r_{i-1}}_\lambda (\textbf{\textup{u}}_{\Delta_\lambda}) = \sum_{j = r_{i-1} + 1}^{r_i} (a + b - 2j) = (r_i - r_{i-1}) (a + b - r_{i-1} - r_i - 1).
\end{equation}
Therefore, by comparing~\eqref{equ_psiabri} and~\eqref{equ_psiabri2}, we derive
$$
\Psi^{r_i, a+b-r_i}_\lambda (\textbf{\textup{u}}) - \Psi^{r_{i-1}, a+b-r_{i-1}}_\lambda (\textbf{\textup{u}}) = \Psi^{r_i, a+b-r_i}_\lambda (\textbf{\textup{u}}_{\Delta_\lambda}) - \Psi^{r_{i-1}, a+b-r_{i-1}}_\lambda (\textbf{\textup{u}}_{\Delta_\lambda})
$$
for each $i$. In particular,~\eqref{equ_traces} follows$\colon$
\begin{align*}
\Psi^{a,b}_\lambda (\textbf{\textup{u}}) &= \sum_{i=1}^s \left(\Psi^{r_i, a+b-r_i}_\lambda (\textbf{\textup{u}}) - \Psi^{r_{i-1}, a+b-r_{i-1}}_\lambda (\textbf{\textup{u}})\right) \\
&=  \sum_{i=1}^s \left( \Psi^{r_i, a+b-r_i}_\lambda (\textbf{\textup{u}}_{\Delta_\lambda}) - \Psi^{r_{i-1}, a+b-r_{i-1}}_\lambda (\textbf{\textup{u}}_{\Delta_\lambda}) \right) = \Psi^{a,b}_\lambda(\textbf{\textup{u}}_{\Delta_\lambda}).
\end{align*}

Conversely, suppose that~\eqref{equ_traces} holds for all $(a,b) \in \mcal{I}_f$.
For $\textbf{\textup{u}} \in f$, it suffices to check that 
\begin{equation}\label{equ_conversegoal}
	\Phi^{a_i,b_i}_\lambda(\textbf{\textup{u}}_{\Delta_\lambda}) = \Phi^{a_i,b_i}_\lambda (\textbf{\textup{u}})
\end{equation}
for any $(a_i, b_i) \in \mcal{I}_f$ in order to prove that $\textbf{\textup{u}}$ is the center of $f$. 
It is because~\eqref{equ_center_poly} yields 
$$
\Phi^{a_i,b_i}_\lambda (\textbf{\textup{u}}) = \Phi^{a_i,b_i}_\lambda(\textbf{\textup{u}}_{\Delta_\lambda}) = b_i - a_i
$$ 
for any $(a_i, b_i) \in \mcal{I}_f$. 
We shall use the induction on $i$ to verify~\eqref{equ_conversegoal}.

Before starting the induction, we observe the following. For any point $\textbf{\textup{u}}$ in the face $f$, we have
\begin{equation}\label{equAAA}
\Phi^{r_{i-1} + 1,a + b - r_{i-1} - 1}_\lambda (\textbf{\textup{u}}) = \cdots = \Phi^{r_{i}, a + b - r_{i}}_\lambda (\textbf{\textup{u}}).
\end{equation} 
If $\left\{ {(j, a + b - j)} ~|~ j =r_{i-1} + 1, \cdots, r_i \right\}$ is contained in an \emph{unbounded} region, our choice $\lambda$ in~\eqref{equation_lambda_monotone} determines $\Phi^{r_{i-1} + 1,a + b - r_{i-1} - 1}_\lambda (\textbf{\textup{u}}) = \cdots = \Phi^{r_{i}, a + b - r_{i}}_\lambda (\textbf{\textup{u}}) = a + b - r_{i-1} - r_i - 1$ because the indices lie in a symmetric region with respect to the line with slope one through $a + b - r_{i-1} - r_i - 1$. In particular,
\begin{align}\label{equ_baseunboundedcase}
\begin{split}
\Psi^{r_i, a+ b - r_i}_\lambda  (\textbf{\textup{u}}) - \Psi^{r_{i-1}, a+ b - r_{i-1}}_\lambda (\textbf{\textup{u}}) &= (r_i - r_{i-1}) \cdot (a + b - r_{i-1} - r_i - 1) \\
&= \sum_{j = r_{i-1} + 1}^{r_i} (a + b - 2j) = \Psi^{r_i, a+b-r_i}_\lambda (\textbf{\textup{u}}_{\Delta_\lambda}) - \Psi^{r_{i-1}, a+b-r_{i-1}}_\lambda (\textbf{\textup{u}}_{\Delta_\lambda}).
\end{split}
\end{align}

We turn to the base case of the induction. With respect to the order defined in~\eqref{equation_labeling}, $(a_1, b_1)$ is the first index contained in the bounded region and hence
$$
 \Phi^{a_1,b_1}_\lambda (\textbf{\textup{u}})  + \Psi^{a_1 - 1,b_1 + 1}_\lambda (\textbf{\textup{u}}) = \Psi^{a_1,b_1}_\lambda (\textbf{\textup{u}}) =
\Psi^{a_1,b_1}_\lambda(\textbf{\textup{u}}_{\Delta_\lambda})  = \Phi^{a_1,b_1}_\lambda(\textbf{\textup{u}}_{\Delta_\lambda}) + \Psi^{a_1-1,b_1+1}_\lambda(\textbf{\textup{u}}_{\Delta_\lambda}).
$$
The previous identity~\eqref{equ_baseunboundedcase} yields that $\Psi^{a_1 - 1,b_1 + 1}_\lambda (\textbf{\textup{u}}) = \Psi^{a_1-1,b_1+1}_\lambda(\textbf{\textup{u}}_{\Delta_\lambda})$ so that 
$$
	\Phi^{a_1,b_1}_\lambda (\textbf{\textup{u}}) = \Phi^{a_1,b_1}_\lambda(\textbf{\textup{u}}_{\Delta_\lambda})  = b_1 - a_1.
$$

We now proceed to the induction step. 
For any $(a := a_\ell, b := b_\ell) \in \mcal{I}_f$ with $\ell > 1$ and any point $\textbf{\textup{u}}$ in the face $f$,
\begin{equation}\label{equAAA}
\Phi^{r_{i-1} + 1,a + b - r_{i-1} - 1}_\lambda (\textbf{\textup{u}}) = \cdots = \Phi^{r_{i}, a + b - r_{i}}_\lambda (\textbf{\textup{u}}).
\end{equation}
Set $A$ to be the value of~\eqref{equAAA}.
If $\left\{ {(j, a + b - j)} ~|~ j =r_{i-1} + 1, \cdots, r_i \right\}$ is contained in an \emph{unbounded} region, it follows from the previous paragraph that $A = a + b - r_{i-1} - r_i - 1$.
If $\left\{ {(j, a + b - j)} ~|~ j =r_{i-1} + 1, \cdots, r_i \right\}$ is contained in a \emph{bounded} region, then there must be a vertex $(a_k, b_k) \in \mcal{I}_f$ (in a \emph{previous induction} step, that is, $k < \ell$) such that its upper rightmost vertex is $(a_k, b_k)$. By symmetry of the region, we have 
$$
a + b - r_{i-1} - r_i - 1 = b_k - a_k = \Phi_\lambda^{a_k, b_k} (\textbf{\textup{u}}_{\Delta_\lambda})
$$

Moreover, because of~\eqref{equ_psiabri2}, 
\begin{equation}\label{equpsiaaa}
\Psi^{r_i, a+b-r_i}_\lambda (\textbf{\textup{u}}_{\Delta_\lambda}) - \Psi^{r_{i-1}, a+b-r_{i-1}}_\lambda (\textbf{\textup{u}}_{\Delta_\lambda}) = (r_i - r_{i-1}) (a + b - r_{i-1} - r_i - 1) = (r_i - r_{i-1}) \cdot \Phi_\lambda^{a_k, b_k}(\textbf{\textup{u}}_{\Delta_\lambda})
\end{equation}
By the induction hypothesis at $(a_k, b_k)$, 
\begin{equation}\label{equ_inductionhypo}
\Phi_\lambda^{a_k, b_k}(\textbf{\textup{u}}_{\Delta_\lambda}) = \Phi_\lambda^{a_k, b_k}(\textbf{\textup{u}})
\end{equation}
Combining~\eqref{equ_traces},~\eqref{equAAA},~\eqref{equpsiaaa}, and~\eqref{equ_inductionhypo}, we obtain
\begin{align*}
(r_i - r_{i-1}) \cdot A &= \Psi^{r_i, a+b-r_i}_\lambda (\textbf{\textup{u}}) - \Psi^{r_{i-1}, a+b-r_{i-1}}_\lambda (\textbf{\textup{u}}) \\
&= \Psi^{r_i, a+b-r_i}_\lambda (\textbf{\textup{u}}_{\Delta_\lambda}) - \Psi^{r_{i-1}, a+b-r_{i-1}}_\lambda (\textbf{\textup{u}}_{\Delta_\lambda}) \\
&= (r_i - r_{i-1}) \cdot \Phi_\lambda^{a_k, b_k}(\textbf{\textup{u}}) = (r_i - r_{i-1}) \cdot (b_k - a_k)
\end{align*}
and therefore $A = b_k - a_k$ as desired. 
Furthermore,~\eqref{equ_psiabri} and~\eqref{equpsiaaa} yield that
\begin{equation}\label{equinductionpreab}
\Psi^{r_i, a+b-r_i}_\lambda (\textbf{\textup{u}}_{\Delta_\lambda}) - \Psi^{r_{i-1}, a+b-r_{i-1}}_\lambda (\textbf{\textup{u}}_{\Delta_\lambda}) = \Psi^{r_i, a+b-r_i}_\lambda (\textbf{\textup{u}}) - \Psi^{r_{i-1}, a+b-r_{i-1}}_\lambda (\textbf{\textup{u}})
\end{equation}
for any $i$.
By~\eqref{equ_traces}, we obtain
$$
\Phi^{a_\ell,b_\ell}_\lambda (\textbf{\textup{u}}_{\Delta_\lambda}) + \Psi^{a_\ell - 1,b_\ell + 1}_\lambda (\textbf{\textup{u}}_{\Delta_\lambda})  = \Psi^{a_\ell,b_\ell}_\lambda (\textbf{\textup{u}}_{\Delta_\lambda}) = \Psi^{a_\ell,b_\ell}_\lambda (\textbf{\textup{u}}) = \Phi^{a_\ell,b_\ell}_\lambda (\textbf{\textup{u}}) + \Psi^{a_\ell -1,b_\ell+1}_\lambda (\textbf{\textup{u}}).
$$
\eqref{equ_conversegoal} follows from~\eqref{equinductionpreab}.
\end{proof}

We are going to use circle actions generated by partial traces~\eqref{equ_partiaalttrr} to obtain gradient holomorphic discs and then apply Corollary~\ref{corollary_Maslov_index_formula} to compute their Maslov indices. 
The action is \emph{not} globally defined because a partial trace is \emph{not} a global smooth function in general. So, we need to confirm that a constructed disc generated by the partial trace is contained in its smooth locus. Moreover, to apply Corollary~\ref{corollary_Maslov_index_formula}, we will check the semifreeness of the action. 

Let $f$ be a face of $\Delta_\lambda$ and $\gamma_f$ be the corresponding subgraph of $\Gamma_\lambda$. 
For each index $(a_i, b_i) \in \mcal{I}_f$, where $\mcal{I}_f$ is given in \eqref{equ_doubleindices}, 
consider a symplectic manifold $(\mcal{U}_{\lambda,-}^{a_i, b_i}, \omega_\lambda)$ with periodic Hamiltonian given by the partial trace in Lemma~\ref{lemma_Ham_action}$\colon$ 
\begin{equation}\label{equ_pselambda}
	\Psi_\lambda^{a_i, b_i} = \Phi_\lambda^{1,a_i+b_i - 1} + \cdots + \Phi_\lambda^{a_i, b_i}.
\end{equation}
The following lemma tells us that the fiber over any point $\textbf{u}$ in the relative interior of $f$ is contained in the smooth locus of $\Psi_\lambda^{a_i, b_i}$. 

\begin{lemma}\label{lemma_smooth_on_Lag}
	The inverse image $\Phi_\lambda^{-1}(\mathring{f})$ is contained in $\mcal{U}_{\lambda,-}^{a_i, b_i}$. 
\end{lemma}

\begin{proof}
	By our choice of the upper rightmost vertex $(a_i,b_i) \in \mcal{I}_f$ in~\eqref{equ_doubleindices}, the edge connecting $(a_i, b_i)$ and $(a_i, b_i - 1)$ is contained in the face $\gamma_f$, see Figure~\ref{figure_aibiineq}.
	Recalling how $\Psi$ in~\eqref{equ_orderpreservingmap} is defined, 
	the presence of the edge implies that any point $\textbf{\textup{u}}$ in the relative interior of $f$ satisfies 
	$u_{a_i, b_i} > u_{a_i + 1, b_i}$. By the GC pattern~\eqref{equation_GC-pattern}, it follows that $u_{a_i + 1, b_i} \geq u_{a_i + 1, b_i - 1}$. For any $\textbf{\textup{u}} \in \mathring{f}$, we have $u_{a_i, b_i} > u_{a_i + 1, b_i - 1}$ and hence $\Phi_\lambda^{-1}(\textbf{\textup{u}})$ is contained in $\mcal{U}_{\lambda,-}^{a_i,b_i}$. 	
\end{proof}
	\begin{figure}[h]
		\vspace{-0.2cm}
		\scalebox{0.85}{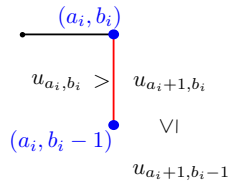}
		\vspace{-0.8cm}
		\caption{\label{figure_aibiineq} The pattern around the vertex $(a_i,b_i)$}
	\end{figure}

A free orbit generated by $\Psi_\lambda^{a_i, b_i}$ and contained in $\Phi_\lambda^{-1}(\textbf{\textup{u}})$ for any $\textbf{\textup{u}} \in \mathring{f}$ can be always taken due to the following. 

\begin{lemma}\label{lemma_free_orbit}\label{lemma_exist_freeorbit}
	Consider an effective symplectic $S^1$-action on a $2n$-dimensional, possibly open,
	 symplectic manifold $(M,\omega)$ and suppose that $L \subset M$ is an $S^1$-invariant closed Lagrangian submanifold of $(M,\omega)$.
	Then there exists at least one free $S^1$-orbit in $L$. 
\end{lemma}

\begin{proof}
	Since $L$ is closed, there are only finitely many orbit types, namely types of $\Z_{p_1}, \cdots,$ and $\Z_{p_r}$ for some positive integers $p_1, \cdots, p_r$ greater than one. 
	Also, since the action is effective and symplectic, each $\Z_{p_i}$-fixed point set denoted by $Z_i$ is a symplectic submanifold of $(M,\omega)$ and has dimension less than $2n$. 
	Thus $Z_i \cap L$ has positive codimension in $L$. Therefore, we have 
	$
		L - \bigcup_{i=1}^r Z_i \neq \emptyset
	$
	and this finishes the proof. 
\end{proof}

Now, we figure out the maximal fixed component of the circle action generated by $\Psi_\lambda^{a_i,b_i}$ and compute its dimension for the purpose of applying Corollary~\ref{corollary_Maslov_index_formula}. 

\begin{lemma}\label{lemma_codim}
	Let $\lambda = \{\lambda_1, \cdots, \lambda_n\}$ be given in \eqref{equation_lambda_monotone}.
	For any vertex $(a,b)$ in $\Gamma_\lambda$, consider the circle action on $\mcal{U}_{\lambda, -}^{a,b}$ generated by $\Psi^{a,b}_\lambda$ in~\eqref{equ_psiijpartial}.
	Then there exists a maximal fixed component, denoted by $Z_\lambda^{a, b}$, of the action contained in $\mcal{U}_{\lambda,-}^{a,b}$.
	Moreover, $Z_\lambda^{a,b}$ is connected and 
	\begin{equation}\label{equ_codimensionformu}
		\mathrm{codim}_\R Z_\lambda^{a, b} = 2 \left( \sum_{i=1}^a \lambda_i - \Psi^{a,b}_\lambda(\textbf{\textup{u}}_{\Delta_\lambda}) \right)
	\end{equation}
	where $\textbf{\textup{u}}_{\Delta_\lambda}$ is the center of ${\Delta_\lambda}$. 
\end{lemma}

\begin{proof}
	By the GC pattern~\eqref{equation_GCdef}, the maximal value of the component $\Phi_\lambda^{i,j}$ is $\lambda_i$ so that the partial trace $\Psi_\lambda^{a, b}$ is bounded by
	$
		\lambda_1 + \cdots + \lambda_{a}
	$ on $\mcal{O}_\lambda$.
	Setting $\widetilde{n}_{k+1} := \min(a, n_{k+1})$,
	consider the set $F$ given by
	\begin{equation}\label{equ_faceFF}
		F = \bigcap_{k=0}^s  \bigcap_{j = a+b - \widetilde{n}_{k+1}}^{n - {n}_{k+1}} \left[ \bigcap_{i = n_k+1}^{\widetilde{n}_{k+1} - 1} 
		 \left\{ \textbf{\textup{u}} \in \R^{\dim \Delta_\lambda} ~\big{|}~ u_{i,j} = u_{i+1, j} \right\} \cap \bigcap_{i = n_k+1}^{\widetilde{n}_{k+1}} 
		 \left\{ \textbf{\textup{u}} \in \R^{\dim \Delta_\lambda} ~\big{|}~ u_{i,j} = u_{i, j+1} \right\} \right]	
	\end{equation}
	where $s$ is the integer satisfying $n_{s} < a \leq n_{s+1}$. (Keep the second picture of Figure~\ref{figure_max} in mind as an example).
	The set $F$ corresponds to a face in $\Gamma_\lambda$, which cut $\Gamma_\lambda$ into $(s+1)$ unbounded regions and several unit blocks$\colon$ 
	\begin{itemize}
		\item The unbounded regions are bounded by 
		\begin{enumerate}
		\item	
		 line segments $\overline{(n_0, a+b-1 - \widetilde{n}_{1}),(\widetilde{n}_{1}, a+b-1 - \widetilde{n}_{1})}$ and 
		$\overline{(\widetilde{n}_{1}, a+b-1 - \widetilde{n}_{1}),(\widetilde{n}_{1}, n - {n}_{1})}$,
		or 
		\item line segments
			$
		\overline{(n_k, a+b-1 - \widetilde{n}_{k+1}),(n_k, n - n_k)}$, 
		$\overline{(n_k, a+b-1 - \widetilde{n}_{k+1}),(\widetilde{n}_{k+1}, a+b-1 - \widetilde{n}_{k+1})}$, and
		$\overline{(\widetilde{n}_{k+1}, a+b-1 - \widetilde{n}_{k+1}),(\widetilde{n}_{k+1}, n - {n}_{k+1})}
	$ for 	$1 \leq k \leq s$.
		\end{enumerate}
		\item The other regions are unit blocks.
	\end{itemize}
	Theorem~\ref{theorem_ACK} asserts that $F$ is indeed a face of $\Delta_\lambda$.
	Again by~\eqref{equation_GCdef}, $F$ can be simply described as the intersection
	\begin{equation}\label{equ_facess}
		F = \bigcap_{i=1}^a \left\{ \textbf{\textup{u}} \in \Delta_\lambda ~\big{|}~ u_{i, a + b - i} = \lambda_i \right\}.
	\end{equation}
	
	For any $\textbf{u} \in F$, the function value $\Psi_\lambda^{a,b} (\textbf{u})$ is exactly $\lambda_1 + \cdots + \lambda_a$ so that the maximum of $\Psi_\lambda^{a,b}$ is achieved exactly on $F$. 
	Moreover, because $u_{a, b} > u_{a + 1, b - 1}$ for every point $\textbf{\textup{u}} = (u_{i,j})$ in the relative interior $\mathring{F}$ of $F$, $\Phi_\lambda^{-1}(\mathring{F})$ is contained in $\mcal{U}_{\lambda, -}^{a,b}$. In particular, the intersection $\mcal{U}_{\lambda, -}^{a,b} \cap \Phi_\lambda^{-1}(F)$ is \emph{non-empty}. In sum, we have a non-empty maximal component $ Z_\lambda^{a,b}$
	\begin{equation}\label{equ_maxiamalcom}
		Z_\lambda^{a,b} = \mcal{U}_{\lambda, -}^{a,b} \cap \Phi_\lambda^{-1}(F).
	\end{equation}
	The connectedness of $Z_\lambda^{a,b}$ follows from \cite[Corollary IV.3.2]{Au}. 
	 	
	Let $\gamma_F$ be the face of $\Gamma_\lambda$ corresponding to $F$ in Theorem \ref{theorem_ACK}, see the second picture of Figure~\ref{figure_max} as an example.
	The filling of the face $\gamma_F$ does \emph{not} contain $L_k$-blocks $(k \geq 2)$.
	We cannot fill any $L$-blocks obeying the condition~\eqref{equation_L_block} into the unbounded regions above the horizontal line segments
	$$
		\left\{ \overline{(n_k, a+b-1 - \widetilde{n}_{k+1}),(\widetilde{n}_{k+1}, a+b-1 - \widetilde{n}_{k+1})} ~|~ 0 \leq k \leq s \right\}.
	$$
	The other regions of $\gamma_f$ can be filled by $L_1$-blocks. 
			
	By Corollary~\ref{prposition_torifactors}, the fiber over any point in $\mathring{F}$ is a $(\dim F)$-dimensional torus. We then obtain 	
	\begin{equation*}
		\mathrm{codim}_\R ~Z_\lambda^{a, b} = 2(\dim~\Delta_\lambda - \dim ~F).
	\end{equation*}
	On the other hand, we have 
	\begin{align*}
		 \sum_{i=1}^a \lambda_i - \Psi^{a,b}_\lambda(\textbf{\textup{u}}_{\Delta_\lambda}) 
		  &=  \displaystyle \sum_{j=1}^{a} \left(\lambda_j - \Phi^{j,a + b - j}_\lambda(\textbf{\textup{u}}_{\Delta_\lambda}) \right) 
			= \displaystyle \sum_{j=1}^{a} \left(\lambda_j - (a + b - j - j) \right),
	\end{align*}
	where the equalities come from ~\eqref{equ_center_poly} and~\eqref{equ_pselambda}. 
	Set $n_0 = 0$.
	From ~\eqref{equation_lambda_monotone}, it follows that
	\begin{align*}
		 \sum_{i=1}^a \lambda_i - \Psi^{a,b}_\lambda(\textbf{\textup{u}}_{\Delta_\lambda}) 
		 &= \sum_{i=0}^s \sum_{j = n_{i} + 1}^{\widetilde{n}_{i+1}} \left( n - n_{i} - n_{i+1} - (a+b) + 2j \right) \\
		 &= \sum_{i=0}^s \left( \left( \widetilde{n}_{i+1} - n_i \right) \cdot \left(n - n_{i} - n_{i+1} - (a+b) \right) +  \sum_{j = n_{i} + 1}^{\widetilde{n}_{i+1}}  2j \right) \\
		 &= \sum_{i=0}^s \left( \left( \widetilde{n}_{i+1} - n_i \right) \cdot \left(n - n_{i} - n_{i+1} - (a+b) + n_i + \widetilde{n}_{i+1} + 1 \right) \right) \\
		 &= \sum_{i=0}^{s-1} \left[ (n_{i+1} - n_i) \cdot \left( \left( n  - n_{i+1} \right) - \left( a+b -n_{i+1} - 1  \right) \right) \right] +  (a - n_{s})\left((n - n_{s+1}) - (b -1) \right)
	\end{align*}
	where $s$ is the integer satisfying $n_{s} < a \leq n_{s+1}$. 
	Notice that the last expression is exactly the number of unit boxes contained in the $(s+1)$ unbounded regions. 
	For instance, $\left[ (n_{1} - n_0) \cdot \left( \left( n  - n_{1} \right) - \left( a+b -n_{1} - 1  \right) \right) \right]$ is the number of unit boxes in the unbounded region bounded by two lines $\overline{(n_0, a+b-1 - n_{1}),({n}_{1}, a+b-1 - {n}_{1})}$ and 
		$\overline{({n}_{1}, a+b-1 - {n}_{1}),({n}_{1}, n - {n}_{1})}$.
	Thus, by Theorem~\ref{theorem_ACK}, we obtain	
	\begin{align*}
			\sum_{i=1}^a \lambda_i - \Psi^{a,b}_\lambda(\textbf{\textup{u}}_{\Delta_\lambda})  & = \dim~\Delta_\lambda - \dim~F.
	\end{align*}
	Hence,~\eqref{equ_codimensionformu} is established.
\end{proof}

\begin{example}\label{example_codim}
	Let $\lambda = \{8,8,3,3,3,-2,-2,-5,-8,-8\}$. 
	We choose a vertex $(4,3)$ of $\Gamma_\lambda$ and consider 
	\[
		\Psi_\lambda^{4,3} = \Phi_\lambda^{1,6} + \Phi_\lambda^{2,5} + \Phi_\lambda^{3,4} + \Phi_\lambda^{4,3}
	\] 
	as an example.
	Any point $z$ in the maximal component $Z^{4,3}_\lambda$ satisfies
	\[
		\Phi_\lambda^{1,6}(z) = \Phi_\lambda^{2,5}(z) = 8, \quad \Phi_\lambda^{3,4}(z) = \Phi_\lambda^{4,3}(z) = 3.
	\]
	Then the image $\Phi_\lambda(Z^{4,3}_\lambda)$ is contained in the face $F$ where the face $\gamma_F$ of $\Gamma_\lambda$ corresponding to $F$ 
	is given in the second diagram of Figure~\ref{figure_max} by the min-max principle. 
	By Theorem~\ref{theorem_ACK} and Corollary~\ref{prposition_torifactors}, we have
	$$
		\mathrm{codim}_\R \Phi^{-1}_\lambda (\mathring{F}) = \mathrm{codim}_\R Z^{3,4}_\lambda = 28.
	$$
	On the other hand, we have 
	\begin{itemize}	
		\item $\sum_{i=1}^4 \lambda_i = 8 + 8 + 3 + 3 = 22$,
		\item $\Psi^{3,4}_\lambda(\textbf{\textup{u}}_{\Delta_\lambda}) = 5 + 3 + 1 + (-1) = 8$.
	\end{itemize}
	Therefore~\eqref{equ_codimensionformu} holds for $(a,b)= (4,3)$.

	\vspace{0.2cm}	
	\begin{figure}[h]
		\scalebox{0.85}{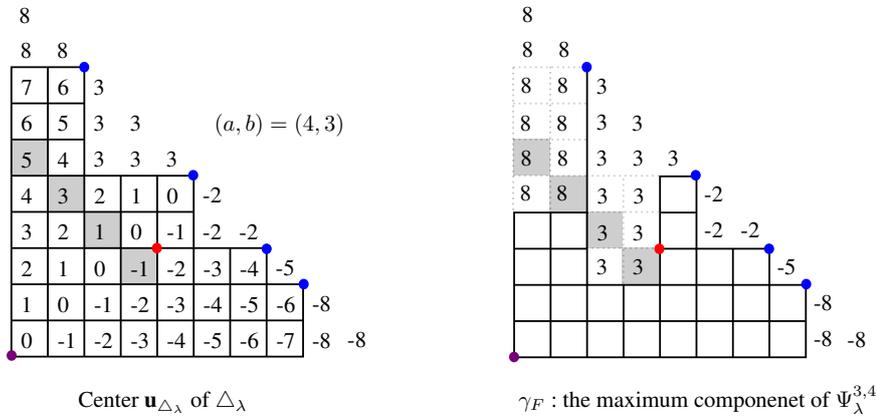}
		\caption{\label{figure_max} Maximal component for $\Psi_\lambda^{4,3}$}
	\end{figure}
\end{example}

\begin{proposition}\label{proposition_semifreee}
	The $S^1$-action on $\mcal{U}_{\lambda, -}^{a,b}$ generated by $\Psi_\lambda^{a,b}$ is semifree near $Z_{\lambda}^{a, b}$.
\end{proposition}

	Before proving Proposition~\ref{proposition_semifreee}, we recall some well-known facts needed for the proof. 
	For any (complex) $n$-dimensional projective toric variety $X$ with a projective embedding $\iota \colon X \rightarrow \p^N$,
	we denote by $T$ and $\frak{t}$ the compact torus in $(\C^*)^n$ acting on $X$ holomorphically and its Lie algebra, respectively.  
	Then a moment map $\mu_X \colon X \rightarrow \frak{t}^* \cong \R^n$ with respect to the K\"{a}hler form induced from the Fubini--Study form on $\p^N$
	is the restriction of a moment map $\mu_{\p^N}$ for the linearly extended Hamiltonian $T$-action on $\p^N$ 
	to $\iota(X)$.
	 Moreover, the image $\mu_X(X)$ is a convex polytope $\Delta_X$
	by the Atiyah--Guillemin--Sternberg convexity theorem \cite{At, GS2}. Let $\ell \in \frak{t}$ be any primitive integral vector that generates a circle subgroup of $T$ and let $e$ be 
	an edge of $\Delta_X$. Then the inverse image $\mu_X^{-1}(e)$ is a $T$-invariant 2-sphere and 
	the order of the isotropy subgroup of a point in $\mu_X^{-1}(\mathring{e})$ for the $S^1$-action 
	is equal to $| \langle \ell, \vec{e} \rangle |$ where $\vec{e}$ denotes 
	a primitive edge vector of $e$.
	
	Also, we employ the following toric degeneration constructed by Nishinou--Nohara--Ueda \cite{NNU} in order for us to pass our case into the toric case. 	
	\begin{theorem}[\cite{NNU}]\label{theorem_NNU}
	There exists a toric degeneration $\pi \colon \mcal{X} \rightarrow \C$ such that 
	\begin{itemize}
		\item the central fiber $X_0 = \pi^{-1}(0)$ is the toric variety associated to $\Delta_\lambda$, and
		\item For $X_1 = \pi^{-1}(1)$ is isomorphic to $\mcal{O}_\lambda$ as a complex manifold.
	\end{itemize}
	Furthermore, there is a continuous map $\phi \colon X_1 \rightarrow X_0$ making the diagram
	 \begin{equation}\label{equation_preserve_integrable_system}
		    \xymatrix{
			      X_1  \ar[dr]_{\Phi_\lambda} \ar[rr]^{\phi}
			      & & X_0 \ar[dl]^{\mu} \\ 
			      & \Delta_{\lambda} &
			    }
	  \end{equation}
	 commute where $\mu = (\mu^{i,j})$ is a moment map on $X_0$ and $\mu^{i,j}$ denotes the component corresponding the $(i,j)$-th coordinates of $\R^{\frac{n(n-1)}{2}}$
	 given in \eqref{equation_GC-pattern}.
	 In particular, we have 
	 $\mu^{i,j} \circ \phi = \Phi_\lambda^{i,j}$ so that the restriction 
	 \[
	 	\phi \colon \mcal{U}_{\lambda, -}^{i,j} \rightarrow X_0
	 \]	
	 is equivariant under the $S^1$-action on $\mcal{U}_{\lambda, -}^{i,j}$ generated by $\Phi_\lambda^{i,j}$ and the $S^1$-action on $X_0$ generated by $\mu^{i,j}$.
	\end{theorem}

	To verify Proposition~\ref{proposition_semifreee}, it is sufficient to show that the $S^1$-action generated by $\Psi_\lambda^{a, b}$ is semifree at some fixed point in the maximal fixed component $Z_\lambda^{a,b}$ since it is connected by Lemma \ref{lemma_codim}. To show that the action is semifree, we will take one particular vertex.
	
	\begin{definition}
		For a given $\lambda$ in~\eqref{equation_lambda_monotone} and any lattice point ${(a,b)}$, let $s$ be the largest integer satisfying $n_s < n + 1 - b$ where $n_\bullet$'s are given in~\eqref{nidef}.
	 Consider the face $\gamma^{a,b}_\lambda$ associated with $(a,b)$ of the ladder diagram $\Gamma_\lambda$ such that the vertical segments in $\gamma^{a,b}_\lambda$ are \emph{exactly}
	  \begin{equation}\label{equation_characterize}
		\left\{\overline{(0,0), (0, b-1)}, \overline{(n_1, b-1) (n_1, n-n_1)}, \cdots, \overline{(n_s, b-1)(n_s, n- n_s)} \right\}. 
	\end{equation} 
	Note that such a face always exists since one can draw horizontal line segments including $\overline{(0,b-1),(n_s,b-1)}$ to make a face in Definition~\ref{def_face}.
	The graph is comb-shaped and it does not bound any closed regions so that the corresponding face is a vertex. The corresponding vertex is called a \emph{comb-shaped vertex} and denoted by $v^{a,b}_\lambda$. See Figure~\ref{Fig_combvert} for examples.
	\end{definition}
	
	\begin{figure}[h]
			\scalebox{0.85}{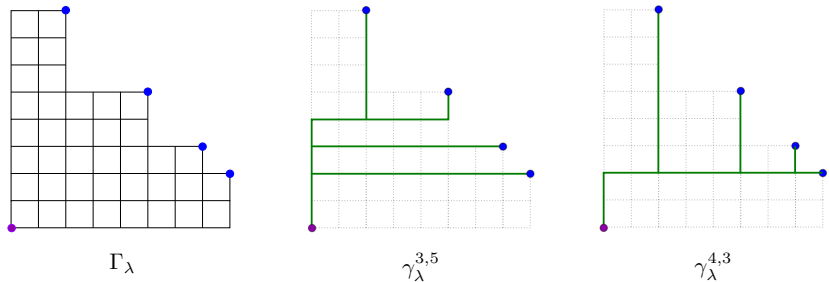}
			\caption{\label{Fig_combvert} Comb-shaped vertices}
	\end{figure}
	
	The comb-shaped vertices enjoy the following properties. 
	
\begin{lemma}\label{lemma_existencefocomb}
	For each lattice point $(a,b) \in \Gamma_\lambda$, the comb-shaped vertex $v^{a,b}_\lambda$ satisfies  
	\begin{enumerate}
		\item The inverse image of $v^{a,b}_\lambda$ is a point contained in $Z_\lambda^{a,b}$.
		\item The vertex $v^{a,b}_\lambda$ is incident to edges 
		$e_1, \cdots, e_{\dim {\Delta_\lambda}}$ such that the fiber over any point in $\mathring{e}_i$ is $S^1$
	\end{enumerate}
\end{lemma}
	
\begin{proof}
	The filling of $\gamma^{a,b}_\lambda$ with $L$-blocks is empty, that is, there does \emph{not} exist any $L$-block satisfying the condition \eqref{equation_L_block}. By Corollary~\ref{prposition_torifactors}, the fiber over $v^{a,b}_\lambda$ is a point. 
Because of presence of the horizontal segment $\overline{(0,b-1),(n_s,b-1)}$ in the comb-shaped vertex, 
we have $u_{a,b} > u_{a, b-1}$ and therefore it must be contained in $\mcal{U}_{\lambda,-}^{a,b}$. Since $\gamma^{a,b}_\lambda$ is contained in $\gamma_F$ corresponding to~\eqref{equ_facess}, $v^{a,b}_\lambda$ is contained in $Z^{a,b}_{\lambda}$ by~\eqref{equ_maxiamalcom} so that $(1)$ is confirmed.  

	Cut $\Gamma_\lambda$ along $\gamma^{a,b}_\lambda$ into several regions. 
	Because of our choice of comb-shaped vertices, the cut regions are all rectangular. 
	Observe that there are two types of cut rectangles$\colon$ the first type contains pieces of $\gamma^{a,b}_\lambda$ in both the bottom edge and the right edge (say $A$-type) and the second type contains pieces of $\gamma^{a,b}_\lambda$ in both the top edge and the left edge (say $B$-type). 
	Since $\gamma^{a,b}_\lambda$ cannot contain any bounded regions, a cut rectangle is of either $A$-type or $B$-type, but not both. For instance, in Figure~\ref{Fig_combvert}, $\gamma_\lambda^{3,5}$ is cut into two $A$-type regions and three $B$-type regions ,and $\gamma_\lambda^{4,3}$ is cut into three $A$-type regions and one $B$-type region.
	
	To construct edges in $(2)$, set
	\begin{itemize}
	\item $\displaystyle
		A_{p,q}(i,j) := \bigcup_{1 \leq t \leq p}  \square^{(i-t+2,j)}	\cup \bigcup_{1 \leq s \leq q}  \square^{(i+1,j + s - 1 )}
	$
	\item $\displaystyle
		B_{p,q}(i,j) := \bigcup_{1 \leq t \leq p}  \square^{(i+t-1,j+1)}	\cup \bigcup_{1 \leq s \leq q}  \square^{(i,j - s + 2)},
	$
	\end{itemize}
	as in  Figure~\ref{figure_pq}.
	
	\begin{figure}[H]
		\scalebox{0.7}{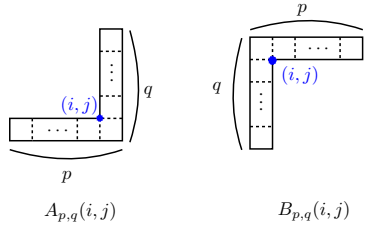}
		\caption{\label{figure_pq} $A_{p,q}(i,j)$ and $B_{p,q}(i,j)$.}
	\end{figure}
	Consider the faces in $\Gamma_\lambda$ which is given by the union of
	\begin{itemize}
		\item The comb-shaped face $\gamma_\lambda^{a,b}$ associated with $(a,b)$ and
		\item The boundary of one \emph{single} block $A_{\bullet,\bullet}(\bullet,\bullet)$ (resp. $B_{\bullet,\bullet}(\bullet,\bullet)$) such that 
		\begin{enumerate}
			\item the rightmost edge and the bottom edge (resp. the leftmost edge and the upper edge) contained in $\gamma_\lambda^{a,b}$
			\item it is fully contained in a single cut rectangle of $A$-type (resp. $B$-type).  
		\end{enumerate}
	\end{itemize}
	See Figure~\ref{figure_vertical_wall} as examples. 
	Since every graph has one bounded region, the corresponding face is one-dimensional.  
	Notice that there are exactly $\dim \Delta_\lambda$ many such faces. 
	Let $\{e_1, \cdots, e_{\dim \Delta_\lambda} \}$ be the corresponding edges of $\triangle_\lambda$.	
	Furthermore, that the generic fiber on each $e_i$ is $S^1$ by Corollary \ref{prposition_torifactors}. 
	\end{proof}	 	
	
	\begin{example}\label{example_edges}
		Let $\lambda$ and $(a,b) = (3,5)$ be given in Figure~\ref{Fig_combvert}. 
		The graph $\gamma^{3,5}_\lambda$ divides $\Gamma_\lambda$ into five regions. 
		There are two $A$-type regions and three $B$-type regions. 
		Figure \ref{figure_vertical_wall} illustrates the faces corresponding to edges $\{e_1, \cdots, e_{\dim \Delta_\lambda}\}$ in Lemma~\ref{lemma_existencefocomb}.  
		\begin{figure}[h]
			\scalebox{0.6}{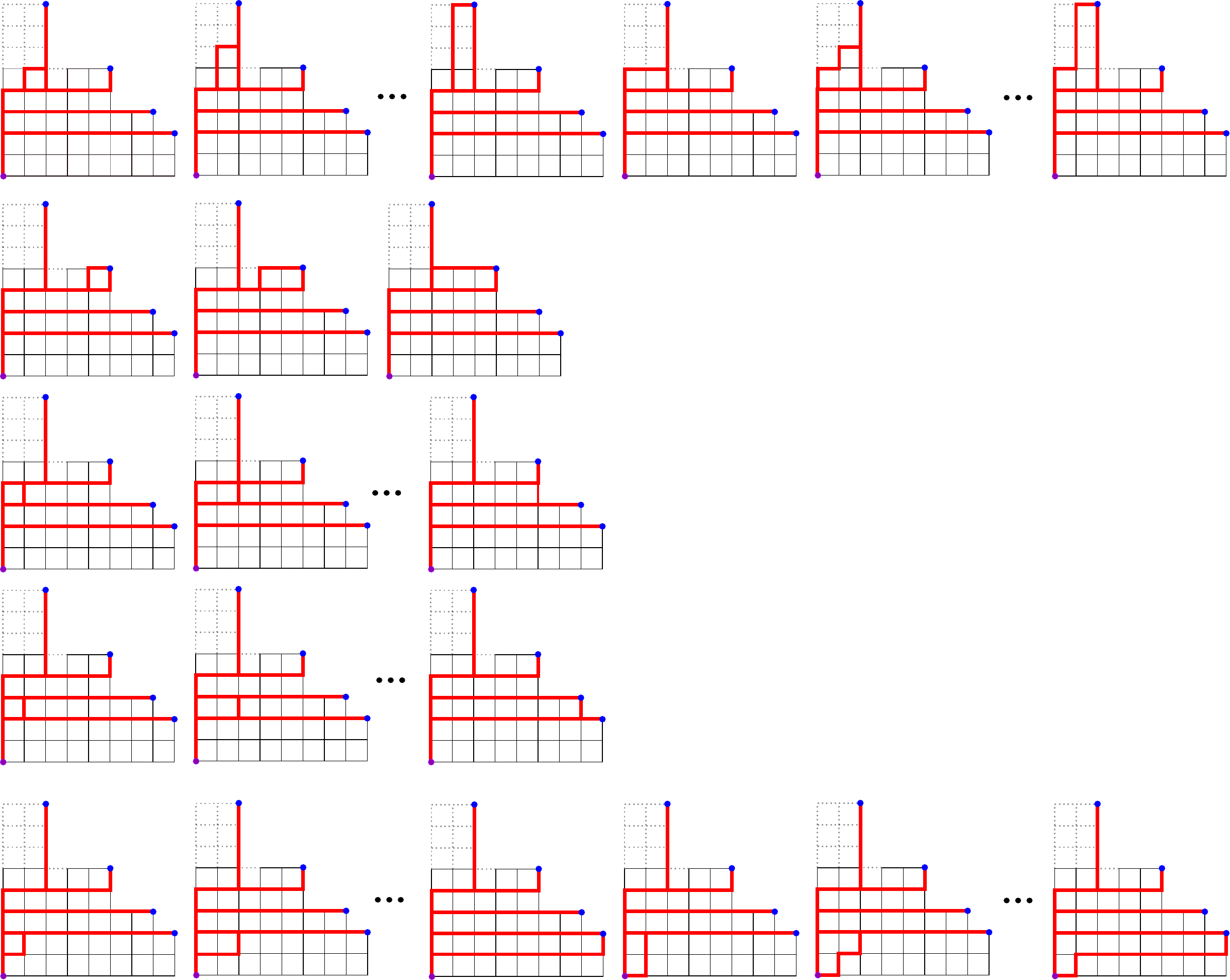}
			\caption{\label{figure_vertical_wall} Edges from Lemma~\ref{lemma_existencefocomb}.}
		\end{figure}
	\end{example}

	 \begin{lemma}\label{sublemma_pointexists}
		Let $\{e_1, \cdots, e_{\dim \triangle_\lambda}\}$ be the set of edges chosen in Lemma~\ref{lemma_existencefocomb}. Set $\vec{e_i}$ to be the primitive edge vector of $e_i$ starting from the comb-shaped vertex $v^{a,b}_\lambda$. Let $\ell = (\ell^{i,j}) \in \frak{t} \cong \R^{\dim \Delta_\lambda}$ be given by 
		\begin{equation}\label{equ_ellell}
			\begin{cases}	
				\ell^{i,j} = 1 & 1 \leq i \leq a, \quad j = a+b-i \\
				\ell^{i,j} = 0 & \text{otherwise}.
			\end{cases}
		\end{equation}
		Then we have
	  	\begin{equation}\label{eq_weightcal}
	  		\langle \ell, \vec{e_i} \rangle = 0 ~\text{or} ~-1 
	  	\end{equation}
	  	for every $i=1,\cdots, \dim \Delta_\lambda$.
	 \end{lemma}
	  	  
	\begin{proof}
	Suppose that the face $\gamma_{e_i}$ corresponding to the edge $e_i$ contains a (single) bounded region. 	Let $\overrightarrow{v_i}$ be the vector such that the components in the bounded region are one and the other components are zero. 
	Note that the edge $e_i$ is parallel to $\overrightarrow{v_i}$. 
	If the single bounded region is of type $A$ (resp. $B$), by the min-max principle, the components in the block of type $A$ decrease (resp. increase) when traveling along the edge $e_i$ from the vertex $v^{a,b}_\lambda$.
	Therefore, $- \overrightarrow{v_i}$ (resp. $\overrightarrow{v_i}$) is the primitive edge vector $\vec{e_i}$ of $e_i$ starting from $v^{a,b}_\lambda$ of type $A$ (resp. $B$).	
	In other words,
		\begin{enumerate}
			\item The $(r,s)$-th component of $\vec{e_i}$ is $-1$ if and only if $\square^{(r,s)}$ is contained in the bounded region of $\gamma_{e_i}$ and the bounded region is of $A$-type. 
			\item The $(r,s)$-th component of $\vec{e_i}$ is $+1$ if and only if $\square^{(r,s)}$ is contained in the bounded region of $\gamma_{e_i}$ and the bounded region is of $B$-type. 
			\item The $(r,s)$-th component of $\vec{e_i}$ is $0$ if and only if $\square^{(r,s)}$ is not contained in the bounded region of $\gamma_{e_i}$. 
		\end{enumerate}	
	
	To construct the edges in Lemma~\ref{lemma_existencefocomb}, we insert a single $A$-type (resp. $B$-type) block into an $A$-type (resp. $B$-type) rectangular cut region.
	Recall that every rectangular region of $A$-type (resp. $B$-type) is above (resp. below) the horizontal segment $\overline{(0,b-1),(n_s,b-1)}$. Since all components $\ell^{i,j}$ for $j < b$ vanish, $\langle \ell, \vec{e_i} \rangle = 0$ if ${e_i}$ is of type $B$. 
	If ${e_i}$ is of type $A$, then $\langle \ell, \vec{e_i} \rangle$ is either $-1$ or $0$ since a component of $\vec{e_i}$ is either $-1$ or $0$. It completes the proof.
\end{proof}	

\begin{example}
	Let us revisit Example~\ref{example_edges}. Consider two faces $\gamma_{e_1}$ and $\gamma_{e_2}$ in Figure~\ref{figure_primve} as an example. The corresponding primitive edge vectors are 
$$
\overrightarrow{e_1} = (a_{i,j}) \mbox{ where $a_{i,j} = $} \begin{cases}
-1 &\quad \mbox{if $(i,j) = (1,5), (2,5), (2,6), (2,7), (2,8)$} \\
0 &\quad \mbox{otherwise}
\end{cases}
$$
and 
$$
\overrightarrow{e_2} = (b_{i,j}) \mbox{ where $b_{i,j} = $} \begin{cases}
1 &\quad \mbox{if $(i,j) = (1,1), (1,2), (2,2)$} \\
0 &\quad \mbox{otherwise}.
\end{cases}
$$
Then $\ell = (\ell^{i,j}) \in \frak{t} \cong \R^{\dim \Delta_\lambda}$ in~\eqref{equ_ellell} with $(a,b) = (3,5)$ satisfy $\langle \ell, \overrightarrow{e_1} \rangle = -1$ and $\langle \ell, \overrightarrow{e_2} \rangle = 0$.
	 \begin{figure}[h]
			\scalebox{0.75}{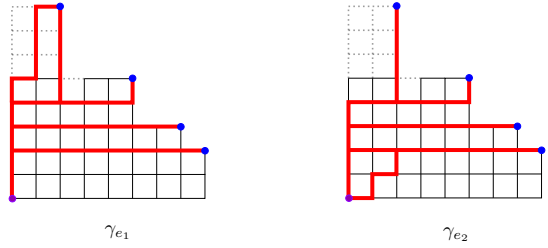}
			\caption{\label{figure_primve} Primitive edge vectors in Lemma~\ref{sublemma_pointexists}.}
	\end{figure}
\end{example}
	
	Now we start the proof of Proposition~\ref{proposition_semifreee}.
	
	\begin{proof}[Proof of Proposition~\ref{proposition_semifreee}]
	By Lemma~\ref{lemma_existencefocomb}, the fiber over the vertex $v^{a,b}_\lambda$ consists of a single point $z$. 
	Note that each Hamiltonian $\Phi_\lambda^{i,j}$ is smooth at the point $z$. 
	The point is a zero of its Hamiltonian vector field. 
	In particular, $z$ is a fixed point of the Hamiltonian $T^n$-action\footnote{The $T^n$-action is generated by the traces of leading principal submatrices
	$\{ \Psi_\lambda^{1,1}, \Psi_\lambda^{2,1}, \cdots, \Psi_\lambda^{n,1} \}$, which are smooth on $\mcal{O}_\lambda$.}
	where $T^n$ is the maximal torus of $U(n)$ acting on $(\mcal{O}_\lambda, \omega_\lambda)$
	in a Hamiltonian fashion. 
	Since the $T^n$-action is holomorphic on $\mcal{O}_\lambda$, the linearized $T^n$-action on $T_z \mcal{O}_\lambda$ is unitary and hence 
	the tangent space $T_z \mcal{O}_\lambda$ is decomposed into one-dimensional $T^n$-representations $\xi_1, \cdots, \xi_{\dim \Delta_\lambda}$
	so that 
	\[
		T_{z} \mcal{O}_\lambda \cong \bigoplus_{i=1}^{\dim \Delta_\lambda} \xi_i 
	\] 
	where $\xi_i$'s are pairwise distinct\footnote{An almost complex manifold $(M,J)$ equipped with a torus action preserving $J$ is called a {\em GKM manifold} if 
	the fixed point set is finite and weights at each fixed point are pairwise disjoint. A partial flag manifold is an example of a GKM manifold, see \cite{GHZ}.}.
	Moreover, each representation space $\xi_i$ is equal to the tangent space $T_z \Phi_\lambda^{-1}(e_i) \cong \C$. Therefore, it is enough to show that 
	the $S^1$-action generated by $\Psi_\lambda^{a,b}$ is semifree on each $T_z \Phi_\lambda^{-1}(e_i)$. 
	Note that by Theorem \ref{theorem_NNU}, 
	the restriction of the map $\phi \colon X_1 \rightarrow X_0$ to $\mcal{U}_{\lambda,-}^{a,b}$ is $S^1$-invariant under the $S^1$-action 
	generated by $\mu^{1, a+b -1} + \cdots + \mu^{a,b}$, a partial trace of the moment map $\mu = (\mu^{i,j}) : X_0 \rightarrow \Delta_\lambda$. 
	Therefore, we only need to check 
	\[
		\langle \ell, \vec{e_i} \rangle = 0 ~\text{or}~ - 1, \quad i=1,\cdots, \Delta_\lambda
	\]
	where $\ell \in \frak{t}$ that generates the circle action generated by $\mu^{1, a+b -1} + \cdots + \mu^{a,b}$. It follows from Lemma~\ref{sublemma_pointexists}. 
	\end{proof}	

\begin{lemma}\label{lemma_fundamental_group}
	Let $f$ be a face of $\Delta_\lambda$ of dimension $k$. 
	For any $\textbf{\textup{u}} \in \mathring{f}$, the fundamental group of $\Phi^{-1}_\lambda(\textbf{\textup{u}})$
	is generated by $\{\sigma_1, \cdots \sigma_k\}$ where the vertices $(a_1,b_1), \cdots, (a_k, b_k)$ of $\Gamma_\lambda$ 
	are defined in \eqref{equation_labeling} and $\sigma_i$ is any free orbit of the $S^1$-action generated by $\Psi_\lambda^{a_i, b_i}$ in $\Phi^{-1}_\lambda(\textbf{\textup{u}})$
	for $i=1,\cdots,k$.
\end{lemma}

\begin{proof}

	Let $f$ be a $k$-dimensional face of $\Delta_\lambda$ and let $\textbf{\textup{u}} = (u_{i,j}) \in \mathring{f}$.
	As we have seen in Section \ref{ssecTopologyOfGelfandCetlinFibers}, $\Phi_\lambda^{-1}(\textbf{\textup{u}})$ is the total space of an $n$-stage iterated bundle 
	\[
		E_n \stackrel{\pi_n} \longrightarrow E_{n-1} \stackrel{\pi_{n-1}}\longrightarrow \cdots \longrightarrow E_2 \stackrel{\pi_2}\longrightarrow 
		E_1  = \mathrm{point}
	\]
	where each $E_l$ is a subset of the set  $\mcal{H}_l$ of $(l \times l)$ Hermitian matrices and $\pi_l \colon E_l \rightarrow E_{l-1}$ is the projection from a $(l \times l)$ Hermitian matrix to 
	its $((l-1) \times (l-1))$ leading principal submatrix.
	
	For each $i=1,\cdots, k$, we first describe the $S^1$-action on $\Phi_\lambda^{-1}(\textbf{\textup{u}})$ generated by $\Phi_\lambda^{a_i, b_i}$ more explicitly. 
	Suppose that $(a_i,b_i)$ is located at the $\ell$-th anti-diagonal, i.e., $a_i + b_i = \ell$, and consider the set 
	\begin{equation}\label{equ_setset}
		\mcal{I}_\ell := \{ (a_j, b_j) ~|~ 1 \leq j \leq i,\, a_j + b_j = \ell \}
	\end{equation}
	consisting of $m$ points where the ordering of $(a_j, b_j)$'s is given in ~\eqref{equation_labeling}. That is, the set $\mcal{I}_\ell$ is equal to 
	\[
		\{ (a_{i-m+1}, b_{i-m+1}), (a_{i-m+2}, b_{i-m+2}), \cdots, (a_i, b_i)\}.
	\]
	
	Now, we consider the filling of $\gamma_f$ with $L$-blocks. 
	Then the $L$-blocks in the filling, whose right vertex of its top edge is located at the $\ell$-th anti-diagonal, 
	correspond to odd dimensional sphere factors of the fiber $F_\ell$ of the projection $\pi_\ell \colon E_\ell \rightarrow E_{\ell-1}$. 
 	We denote by $L_{t_1}, \cdots, L_{t_r}$ the $L$-blocks in the filling of $\gamma_f$ each of which is an $L$-block of size $t_i$ and is located 
 	at $(s_j, l + 1 - s_j - t_j)$ for some $s_1 < \cdots < s_r$.
 	
 	By Corollary 6.10 in \cite{CKO1}, we know that $F_\ell$ is diffeomorphic to 
	\begin{equation}\label{equation_solution}
		F_\ell \cong \{ (z_1, \cdots, z_{\ell-1}) \in \C^{\ell-1} ~|~ \text{$|z_{s_j}|^2 + \cdots + |z_{s_j+t_j-1}|^2 = C_{s_j}$ for some $C_{s_j}>0$} \quad j=1,\cdots, r\}.
	\end{equation}
	In particular, each $(a_j, b_j)$ indicates a location of an $L_1$-block in the filling of $\gamma_f$ and hence we have 
	\[
		\{a_{i-m+1}, a_{i-m+2}, \cdots, a_i \} \subset \{s_1, \cdots, s_r\}.
	\]
	The $S^1$-action on the fiber $F_l$ of $\pi_l$ generated by $\Phi_\lambda^{a_i, b_i}$ is expressed by
	\[
		t \cdot (z_1,\cdots, z_{l-1}) = (z_1, \cdots, z_{a_i-1}, tz_{a_i}, z_{a_i+1}, \cdots, z_{l-1}).
	\]
	Note that the action extends to the total space $E_n = \Phi_\lambda^{-1}(\textbf{\textup{u}})$ such that the projection map $\pi_q$ ($q=2,\cdots,n$) on each stage is $S^1$-equivariant. 
	Moreover, since the action on $E_l$ is fiberwise, the induced action on $E_{l'}$ for $l' < l$ is trivial. 
	Furthermore, as the action is free on $E_l$, so is free on $E_n$.  
	
	Let $\tau_i$ be any free $S^1$-orbit in $\Phi^{-1}_\lambda(\textbf{\textup{u}})$ for the $S^1$-action generated by $\Phi_\lambda^{a_i, b_i}$.
	Note that the homotopy class of $\tau_i$ does not depend on the choice of orbits. 
	By~\eqref{equ_circlefacctt} and Proposition~\ref{theorem_homotopygroup}, 
	$\Phi^{-1}_\lambda(\textbf{\textup{u}}) \cong (S^1)^k \times Y_f$ for some simply connected manifold $Y_f$ and 
	its fundamental group $\pi_1(\Phi^{-1}_\lambda(\textbf{\textup{u}})) \cong \Z^k$ 
	is generated by $\{\tau_1, \cdots, \tau_k \}$.
	Identifying  $\tau_i$ with $(\underbrace{0,\cdots,0}_{i-1}, 1, 0, \cdots, 0) \in \Z^k$, in order to confirm that the fundamenal group generated by $\sigma_i$'s,
	it is enough to show that 
	\[
		\sigma_i = (\underbrace{*, \cdots, *}_{i-1}, 1, 0, \cdots, 0) \in \Z^k \cong \pi_1(\Phi^{-1}_\lambda(\textbf{\textup{u}})), \quad i=1,\cdots,k
	\] 
	where $\sigma_i$ is a free $S^1$-orbit of the $S^1$-action generated by the partial trace $\Psi_\lambda^{a_i, b_i}$. 
		
	The partial trace $\Psi_\lambda^{a_i, b_i}$ in~\eqref{equ_pselambda} can be grouped into two parts
	\begin{align*}
		\Psi_\lambda^{a_i, b_i} &= \sum_{j=1}^{a_i} \Phi_\lambda^{j,\ell - j} = \sum_{(j,\ell - j) \in \mcal{I}_\ell} \Phi_\lambda^{j,\ell - j}  + \sum_{(j,\ell - j) \notin \mcal{I}_\ell} \Phi_\lambda^{j,\ell - j} \\
						&=  \left( \sum_{j=1}^m \Phi_\lambda^{a_{i-j+1}, b_{i-j+1}} \right) + \sum_{(j,\ell - j) \notin \mcal{I}_\ell} \Phi_\lambda^{j,\ell - j} 
	\end{align*}
	where the second summation is over all $j$ such that $1 \leq j \leq a_i$ and $(j, \ell - j) \notin \mcal{I}_\ell$.	
	Note the first summand $\left( \sum_{j=1}^m \Phi_\lambda^{a_{i-j+1}, b_{i-j+1}} \right)$ generates the homotopy class 
	$\sum_{j=1}^m \tau_{i-j+1}$. 
	Moreover, the second summand generates an $S^1$-action, which is trivial on $E_l$ 
	(since the portion of $F_l$ on which the $S^1$ acts in \eqref{equation_solution} is a point)
	but may be extended to a non-trivial action on the total space $E_n$. Consequently, 
	we have 
	\[
		\sigma_i = (\underbrace{*, \cdots, *}_{\ell}, \underbrace{1, \cdots, 1}_{j}, 0, \cdots, 0) \in \Z^k \cong \pi_1(\Phi_\lambda^{-1}(\textbf{\textup{u}})).
	\]
	This completes the proof.
\end{proof}

Finally, we prove our main theorem.

\begin{proof}[Proof of Theorem~\ref{theorem_main}]
	Assume that a Lagrangian GC fiber $\Phi_\lambda^{-1}(\textbf{\textup{u}})$ does not contain any circle factors. Then,~\eqref{equ_circlefacctt} and Proposition~\ref{theorem_homotopygroup} imply that $\pi_1 (\Phi_\lambda^{-1}(\textbf{\textup{u}})) = 0$. Thus, the fiber is monotone since $(\mcal{O}_\lambda, \omega_\lambda)$ is monotone. 
	
	Now consider the case where a Lagrangian GC fiber has circle factors. 
	Let $f$ be a Lagrangian face of dimension $k (>0)$ and $\gamma_f$ the face of $\Gamma_\lambda$ corresponding to $f$. 
	Note that $\dim f$ is the number of bounded regions in $\gamma_f$
	and we denote by $(a_1, b_1), \cdots, (a_k, b_k)$ the upper rightmost vertices of the bounded regions of $\gamma_f$ 	
	respecting the ordering given in \eqref{equation_labeling}. 
	
	We denote by $S^1(a_i, b_i)$ the circle group generated by $\Psi_\lambda^{a_i, b_i}$
	for $i=1,\cdots,k$.
	Let $\textbf{\textup{u}} = (u_{i,j}) \in \mathring{f}$. 
	For the sake of convenience, we denote $\Phi_\lambda^{-1}(\textbf{\textup{u}})$ by $L_\textbf{\textup{u}}$.
	Fix $(a_i,b_i) \in \mcal{I}_f$ where $\mcal{I}_f$ is in~\eqref{equ_doubleindices}.
	By Lemma \ref{lemma_free_orbit} and Lemma \ref{lemma_fundamental_group},
	there exists a free $S^1(a_i,b_i)$-orbit $\sigma_i$ such that 
	$\{\sigma_1, \cdots, \sigma_k\}$ forms a basis of $\pi_1(L_\textbf{\textup{u}}) \cong \Z^k$. 
	
	To construct a disc bounded by $\sigma_i$, we first choose an $\omega_\lambda$-compatible $S^1(a_i, b_i)$-invariant almost complex structure $J$ on $\mcal{U}_{\lambda,-}^{a_i,b_i}$.
	Let $c_i := \Psi_\lambda^{a_i, b_i}(Z_\lambda^{a_i, b_i})$ be the maximum of $\Psi_\lambda^{a_i, b_i}$.
	With respect to the metric $\omega_\lambda(J\cdot, \cdot)$, the gradient flow of $\Psi_\lambda^{a_i, b_i}$ is defined on $\mcal{U}_{\lambda, -}^{a_i, b_i}$. 
	Let $W_i^{{{u}}}$ be the set of unstable points converging to $Z_\lambda^{a_i, b_i}$. 
	
	We then consider the reduced space $(R_{\textbf{\textup{u}}(i)}, \omega_{\textbf{\textup{u}}(i)})$ at level $\textbf{\textup{u}}(i)$ where 
	\[
		R_{\textbf{\textup{u}}(i)} := 
		\left. \left(\Psi_\lambda^{a_i, b_i}\right)^{-1}(\textbf{\textup{u}}(i)) \right/ S^1(a_i, b_i), \quad 	\textbf{\textup{u}}(i) := \Psi^{a_i,b_i}_\lambda(\textbf{\textup{u}}) = u_{1, a_i + b_i - 1} + \cdots + u_{a_i, b_i}
	\]  equipped with the induced symplectic form which we denote by $\omega_{\textbf{\textup{u}}(i)}$. Since $\sigma_i$ is a free orbit, $[\sigma_i]$ is a smooth point 
	in $R_{\textbf{\textup{u}}(i)}$ so that there exists 
	a smooth Darboux neighborhood $U \subset R_{\textbf{\textup{u}}(i)}$ of $[\sigma_i]$. 
	Moreover, we may choose an element
	$[\sigma_i'] \in U \cap \left( W_i^{{u}} / S^1 \right)$ sufficiently close to $[\sigma_i]$ such that 
	\begin{itemize}
		\item there is a smooth path $\gamma$ in $U$ from $[\sigma_i]$ to $[\sigma_i']$, and 
		\item there is a symplectic isotopy $\{\phi_t\}_{0 \leq t \leq 1}$ on $(R_{\textbf{\textup{u}}(i)}, \omega_{\textbf{\textup{u}}(i)})$ 
		such that the support is in $U$, $\phi_0 = \mathrm{id}$, and 
		$\phi_1$ sends $[\sigma_i]$ to $[\sigma_i']$ 
	\end{itemize} 
	(Otherwise, we can find an open neighborhood $U'$ of $[\sigma_i]$ in $U$ such that $U' \cap \left( W_i^{{u}} / S^1 \right) = \emptyset$, which implies that 
	the preimage $\widetilde{U}'$ of $U'$ under the quotient map, which is open in $H^{-1}({\textbf {\textup u}}(i))$, does not intersect $W_i^{{u}}$. Then there is an open 
	neighborhood of $\widetilde{U}'$ in $M$ obtained by flowing $\widetilde{U}'$ along the gradient vector field of $H$ and does not intersection $W_i^{{u}}$. That leads to a contradiction that 
	$W_i^{{u}}$ is open dense.)
	If we denote by $\bar{L}_{\textbf{\textup{u}}}$ the quotient image of $L_{\textbf{\textup{u}}}$ in $R_{\textbf{\textup{u}}(i)}$, then $\phi_1$ maps 
	$\bar{L}_{\textbf{\textup{u}}}$ to some Lagrangian submanifold $\bar{L}_{\textbf{\textup{u}}}'$ containing $[\sigma_i']$. 
	Moreover, $\{\phi_t \}_{0 \leq t \leq 1}$ lifts to a Lagrangian isotopy 
	from a neighborhood of $\sigma_i$ in $L_{\textbf{\textup{u}}}$ to a neighborhood of $\sigma_i'$ in $L_{\textbf{\textup{u}}}'$, the preimage of $\bar{L}_{\textbf{\textup{u}}}'$ 
	under the quotient map.
	
	Finally, consider a disc $u_{\sigma_i} \colon (\mathbb{D}, \pa \mathbb{D}) \rightarrow (\mcal{O}_\lambda, {L}_{\textbf{\textup{u}}})$ obtained by 
	gluing
	\begin{itemize}
		\item  the gradient $J$-holomorphic disc $u^J_{\sigma_i'} \colon (\mathbb{D}, \pa \mathbb{D}) \rightarrow (\mcal{O}_\lambda, {L}_{\textbf{\textup{u}}}^\prime)$ and 
		\item the cylinder $q^{-1}(\gamma)$ where 
	$
		q \colon \left(\Phi_\lambda^{a_i, b_i} \right)^{-1}(\textbf{\textup{u}}(i)) \rightarrow R(\textbf{\textup{u}}(i))
	$
	is the quotient map. 
	\end{itemize}
	
	It remains to show that $\mu([u_{\sigma_i}]) = 2 \int_\mathbb{D} (u_{\sigma_i})^* \omega_\lambda$ for every $i=1,\cdots, k$ if and only if ${\textbf {\textup{u}}}$ is the center of $f$. 
	Remind that $c_1(\mcal{O}_\lambda) = [\omega_\lambda]$ where $\lambda$ is given in~\eqref{equation_lambda_monotone}. 
	We compute the symplectic area and the Maslov index of $u_{\sigma_i}$ as follows.
	\begin{itemize}
		\item (Symplectic area) 
		\[
		\int_\mathbb{D} (u_{\sigma_i})^* \omega_\lambda = \int_\mathbb{D} (u_{\sigma_i'})^* \omega_\lambda = c_i - \textbf{\textup{u}}(i) = c_i -  \Psi^{a_i,b_i}_\lambda(\textbf{\textup{u}})
		\]
		since the symplectic area of the cylinder $q^{-1}(\gamma)$ is zero.
		\item (Maslov index) By Lemma \ref{lemma_isotopy}, Corollary \ref{corollary_Maslov_index_formula}, and Lemma  \ref{lemma_codim}, 
		\[	
			\mu([u_{\sigma_i}]) = \mu([u_{\sigma_i'}]) = \mathrm{codim} ~Z_\lambda^{a_i,b_i}= 2 ~\left(c_i - \Psi^{a_i,b_i}_\lambda(\textbf{\textup{u}}_{\Delta_\lambda})\right).
		\] 
	\end{itemize}
	By Lemma~\ref{lemma_charpartialtraces}, a point $\textbf{u}$ is the center of $f$ if and only if
	$\Psi^{a_i,b_i}_\lambda(\textbf{\textup{u}}_{\Delta_\lambda}) = \Psi^{a_i,b_i}_\lambda (\textbf{\textup{u}})$
	for $i = 1, \cdots, k$. In this case, we have 
	\[
		\mu([u_{\sigma_i}]) =  2 ~\left(c_i - \Psi^{a_i,b_i}_\lambda(\textbf{\textup{u}}_{\Delta_\lambda})\right) = 2 \left(  c_i -  \Psi^{a_i,b_i}_\lambda(\textbf{\textup{u}}) \right) = 
		2 \int_\mathbb{D} (u_{\sigma_i})^* \omega_\lambda
	\] 
	for every $i=1,\cdots, k$. It completes the proof. 
\end{proof}

\bibliographystyle{annotation}

\end{document}